\documentclass{amsart}
\usepackage{amsthm}
\usepackage{amssymb, latexsym}
\usepackage{enumerate}
\usepackage{mathtools,color}

\newcommand\supp{\operatorname{supp}}
\newcommand\card{\operatorname{card}}

\newtheorem{prop}{Proposition}
\newtheorem{thm}[prop]{Theorem}
\newtheorem{rem}{Remark}
\newtheorem{lem}[prop]{Lemma}
\newtheorem{res}{Result}

\begin{document}

\title[Scattering For Loglog Supercritical Radial Schr\"odinger Equation] {Scattering Above Energy Norm Of Solutions Of A Loglog
Energy-Supercritical Schr\"{o}dinger Equation With Radial Data}

\author{Tristan Roy}
\address{Nagoya University}
\email{tristanroy@math.nagoya-u.ac.jp}

\begin{abstract}
We prove scattering of  $\tilde{H}^{k}:= \dot{H}^{k} (\mathbb{R}^{n}) \cap \dot{H}^{1} (\mathbb{R}^{n})$- solutions of the loglog
energy-supercritical Schr\"odinger equation $ i \partial_{t} u + \triangle u = |u|^{\frac{4}{n-2}} u \log^{c} {( \log{(10+|u|^{2})} )} $, $0 < c
< c_{n}$, $n \in \{ 3,4 \}$, with radial data $u(0)=u_{0} \in \tilde{H}^{k}:= \dot{H}^{k} (\mathbb{R}^{n}) \cap \dot{H}^{1} (\mathbb{R}^{n}) $, $k > \frac{n}{2}$. This is achieved, roughly speaking, by extending Bourgain's argument \cite{bour} (see also Grillakis \cite{grill}) and Tao's
argument \cite{taorad} in high dimensions.
\end{abstract}

\maketitle

\section{Introduction}

We shall study the solutions of the following Schr\"odinger equation in dimension $n$, $n \in \{ 3,4 \}$:

\begin{equation}
\begin{array}{ll}
i \partial_{t} u + \triangle u & = |u|^{\frac{4}{n-2}} u g(|u|)
\end{array}
\label{Eqn:BarelySchrod}
\end{equation}
with $g(|u|):= \log^{c}{( \log{ (10 + |u|^{2}) } ) }$, $ 0 <c < c_{n}$ and \footnote{we shall prove global well-posedness
and scattering of radial solutions to (\ref{Eqn:BarelySchrod}). The computations show that these properties
hold for functions $g$ that do not grow faster than $ x \rightarrow \log^{c} \log (10 + |x|^{2})$ with $c < c_n$ but not for functions $g$ that
grow faster (i.e $c \geq c_n$). The values of $c_n$ are determined by technical computations but do not have a
particular physical meaning.}

\begin{equation}
\begin{array}{l}
c_{n}:= \left\{
\begin{array}{ll}
\frac{1}{5772}, \, n=3 \\
\frac{3}{8024}, \, n=4 \\
\end{array}
\right.
\end{array}
\end{equation}
This equation has many connections with the following power-type Schr\"odinger equation, $p>1$

\begin{equation}
\begin{array}{ll}
i \partial_{t} v + \triangle v & = |v|^{p-1} v
\end{array}
\label{Eqn:Schrodpowerp}
\end{equation}
(\ref{Eqn:Schrodpowerp}) has a natural scaling: if $v$ is a solution of (\ref{Eqn:Schrodpowerp}) with data $v(0):=v_{0}$ and if $\lambda \in
\mathbb{R}$ is a parameter then $v_{\lambda}(t,x) := \frac{1}{\lambda^{\frac{2}{p-1}}} v \left( \frac{t}{\lambda^{2}}, \frac{x}{\lambda}
\right)$ is also a solution of (\ref{Eqn:Schrodpowerp}) but with data $v_{\lambda}(0,x):= \frac{1}{\lambda^{\frac{2}{p-1}}} v_{0} \left(
\frac{x}{\lambda} \right)$. If $s_{p}:= \frac{n}{2}- \frac{2}{p-1}$ then the $\dot{H}^{s_{p}}$ norm of the initial data is invariant under the
scaling: this is why (\ref{Eqn:Schrodpowerp}) is said to be $\dot{H}^{s_{p}}$- critical. If $p=1 + \frac{4}{n-2}$ then (\ref{Eqn:Schrodpowerp})
is $\dot{H}^{1}$ (or energy) critical. The energy-critical Schr\"odinger equation
\begin{equation}
\begin{array}{ll}
i \partial_{t} u + \triangle u & = |u|^{\frac{4}{n-2}} u
\end{array}
\label{Eqn:EnergyCrit}
\end{equation}
has received a great deal of attention. Cazenave and Weissler \cite{cazweiss} proved the local well-posedness of (\ref{Eqn:EnergyCrit}): given
any $u(0)$ such that $\| u(0) \|_{\dot{H}^{1}} < \infty$ there exists, for some $t_{0}$ close to zero, a unique $ u \in \mathcal{C} ( [0,t_{0}],
\dot{H}^{1} ) \cap L_{t}^{ \frac{2(n+2)}{n-2}} L_{x}^{\frac{2(n+2)}{n-2}} ( [0,t_{0}] )$ satisfying (\ref{Eqn:EnergyCrit}) in the sense of
distributions

\begin{equation}
\begin{array}{ll}
u(t) & = e^{it \triangle} u(0) - i \int_{0}^{t} e^{i(t-t^{'}) \triangle} \left[  |u(t')|^{\frac{4}{n-2}} u(t') \right] \, dt^{'}
\end{array}
\label{Eqn:DistribSchrod}
\end{equation}
Bourgain \cite{bour} proved global existence and scattering of radial solutions in the class $\mathcal{C} \left( \mathbb{R}, \dot{H}^{1} \right)  \cap L_{t}^{ \frac{2(n+2)}{n-2}} L_{x}^{\frac{2(n+2)}{n-2}} (\mathbb{R})$ in dimension $n=3,4$. He also proved this fact that for smoother solutions. Another proof was given by Grillakis
\cite{grill} in dimension $n=3$. The radial assumption for $n=3$ was removed by Colliander-Keel-Staffilani-Takaoka-Tao
\cite{collkeelstafftaktao}. This result was extended to $n=4$ by Rickman-Visan \cite{rickmanvisan} and to $n \geq 5$ by Visan \cite{visan}. If
$p > 1+ \frac{4}{n-2}$ then $s_{p} > 1$ and we are in the energy supercritical regime. The global existence of $\tilde{H}^{k}$-solutions in this
regime is an open problem. Since for all $\epsilon > 0$ there exists $c_{\epsilon} > 0$ such that $ \left| |u|^{\frac{4}{n-2}} u \right|
\lesssim \left| |u|^{\frac{4}{n-2}} u g(|u|) \right| \leq c_{\epsilon} \max{(1, | |u|^{\frac{4}{n-2}+ \epsilon} u | ) }$ then the nonlinearity
of (\ref{Eqn:BarelySchrod}) is said to be barely supercritical.

In this paper we are interested in establishing global well-posedness and scattering of $\tilde{H}^{k} := \dot{H}^{k}(\mathbb{R}^{n}) \cap \dot{H}^{1}
(\mathbb{R}^{n})$ - solutions of (\ref{Eqn:BarelySchrod}) for $n \in \{ 3,4 \}$. First we prove a local-wellposed result. The local
well-posedness theory for (\ref{Eqn:BarelySchrod}) and for $\tilde{H}^{k}$-solutions can be formulated as follows

\begin{prop}{\textbf{``Local well-posedness ''}}
Let $n \in \{ 3,4 \}$ and $k > \frac{n}{2}$. Let $M$ be such that $\| u_{0} \|_{\tilde{H}^{k}} \leq M$. Then there exists $\delta:= \delta(M) > 0$ small such that if $T_{l}>0$ ($T_{l}$=time of local existence) satisfies

\begin{equation}
\begin{array}{ll}
\| e^{i t \triangle} u_{0}  \|_{L_{t}^{\frac{2(n+2)}{n-2}} L_{x}^{\frac{2(n+2)}{n-2}} ([0,T_{l}]) } & \leq \delta
\end{array}
\end{equation}
then there exists a unique

\begin{equation}
\begin{array}{l}
u \in \mathcal{C}([0,T_{l}], \tilde{H}^{k}) \cap L_{t}^{\frac{2(n+2)}{n-2}} L_{x}^{\frac{2(n+2)}{n-2}} ([0,T_{l}] )
\cap L_{t}^{\frac{2(n+2)}{n}} D^{-1} L_{x}^{\frac{2(n+2)}{n}} ([0,T_{l}]) \\
\cap L_{t}^{\frac{2(n+2)}{n}} D^{-k} L_{x}^{\frac{2(n+2)}{n}}
([0,T_{l}])
\end{array}
\label{Eqn:ClassSol}
\end{equation}
such that

\begin{equation}
\begin{array}{l}
u(t)=e^{i t \triangle} u_{0} - i \int_{0}^{t} e^{i(t- t^{'}) \triangle} \left( |u(t^{'})|^{\frac{4}{n-2}} u(t^{'}) g(|u(t^{'})|) \right) \,
dt^{'}
\end{array}
\label{Eqn:DistribSchrodg}
\end{equation}
is satisfied in the sense of distributions. Here $D^{-\alpha} L^{r}:=\dot{H}^{\alpha, r}$ endowed with the norm $\| f \|_{D^{-\alpha} L^{r}} := \| D^{\alpha} f \|_{L^{r}}$.
\label{Prop:LocalWell}
\end{prop}
This allows to define the notion of maximal time interval of existence $I_{max}$, that is the union of all the open intervals $I$ containing $0$ such
that (\ref{Eqn:DistribSchrodg}) holds in the class $ \mathcal{C}( I, \tilde{H}^{k}) \cap L_{t}^{\frac{2(n+2)}{n-2}} L_{x}^{\frac{2(n+2)}{n-2}}
(I) \cap L_{t}^{\frac{2(n+2)}{n}} D^{-1} L_{x}^{\frac{2(n+2)}{n}} (I) \cap L_{t}^{\frac{2(n+2)}{n}} D^{-k} L_{x}^{\frac{2(n+2)}{n}} (I) $. Next
we prove a criterion for global well-posedness:

\begin{prop}{\textbf{``Global well-posedness: criterion''}}
If $|I_{max}|< \infty $ then

\begin{equation}
\begin{array}{ll}
\| u \|_{L_{t}^{\frac{2(n+2)}{n-2}} L_{x}^{\frac{2(n+2)}{n-2}} (I_{max})} & = \infty
\end{array}
\end{equation}
\label{Prop:GlobWellPosedCrit}
\end{prop}
These propositions are proved in Section \ref{Sec:LocalWell}. With this in mind, global well-posedness follows from an \textit{a priori} bound
of the form

\begin{equation}
\begin{array}{l}
\| u \|_{L_{t}^{\frac{2(n+2)}{n-2}} L_{x}^{\frac{2(n+2)}{n-2}} ([-T,T]) } \leq f(T, \| u_{0} \|_{\tilde{H}^{k}})
\end{array}
\end{equation}
for arbitrarily large time $T>0$. In fact we shall prove that the bound does not depend on time $T$: this is the preliminary step to prove
scattering.

The main result of this paper is:

\begin{thm}
The solution of (\ref{Eqn:BarelySchrod}) with radial data $u(0):=u_{0} \in \tilde{H}^{k}$, $n \in \{ 3, 4\}$, $ k > \frac{n}{2}$ and $0 < c <c_{n}$ exists
for all time $T$. Moreover there exists a scattering state $u_{0,+} \in \tilde{H}^{k} $ such that

\begin{equation}
\begin{array}{ll}
\lim \limits_{t \rightarrow \infty} \| u(t) - e^{it \triangle} u_{0,+}  \|_{\tilde{H}^{k}} & = 0
\end{array}
\end{equation}
and there exists $C$ depending only on $\| u_{0} \|_{\tilde{H}^{k}}$ such that

\begin{equation}
\begin{array}{ll}
\| u \|_{L_{t}^{\frac{2(n+2)}{n-2}} L_{x}^{\frac{2(n+2)}{n-2}} (\mathbb{R}) } & \leq C ( \| u_{0} \|_{\tilde{H}^{k}} )
\end{array}
\end{equation}
\label{thm:main}
\end{thm}

\begin{rem}
This implies global regularity \footnote{By global regularity we mean ``if the data is radial and Schwartz, then the solution is infinitely
differentiable for all time.'' It is well-known that if for all time we have a finite bound of the $L^{\infty}$ norm of the solution,
then we have global regularity.} since by the Sobolev embedding $ \| u \|_{L_{t}^{\infty} L_{x}^{\infty} (\mathbb{R})}
\lesssim \| u \|_{L_{t}^{\infty} \tilde{H}^{k} (\mathbb{R})}$ for $k > \frac{n}{2}$.
\end{rem}

We recall some estimates. The pointwise dispersive estimate is $\| e^{i t \triangle} f  \|_{L^{\infty}(\mathbb{R}^{n})} \lesssim
\frac{1}{|t|^{\frac{n}{2}}} \| f \|_{L^{1} (\mathbb{R}^{n})}$. Interpolating with $\| e^{it \triangle} f \|_{L^{2}(\mathbb{R}^{n})} = \| f
\|_{L^{2} (\mathbb{R}^{n})}$ we have the well-known generalized pointwise dispersive estimate:

\begin{equation}
\begin{array}{ll}
\| e^{it \triangle} f \|_{L^{p}(\mathbb{R}^{n})} & \lesssim \frac{1}{|t|^{n \left( \frac{1}{2} - \frac{1}{p} \right)}} \| f
\|_{L^{p^{'}}(\mathbb{R}^{n})}
\end{array}
\label{Eqn:DispIneq}
\end{equation}
Here $2 \leq p \leq \infty$ and $p^{'}$ is the conjugate of $p$. We recall some useful Sobolev inequalities:

\begin{equation}
\begin{array}{ll}
\| u \|_{L_{t}^{\frac{2(n+2)}{n-2}} L_{x}^{\frac{2(n+2)}{n-2}} (J) } & \lesssim \| D u \|_{L_{t}^{\frac{2(n+2)}{n-2}} L_{x}^{\frac{2
n(n+2)}{n^{2}+4}} (J) }
\end{array}
\label{Eqn:SobolevIneq1}
\end{equation}
and

\begin{equation}
\begin{array}{ll}
\| u \|_{L_{t}^{\infty} L_{x}^{\infty} (J) } & \lesssim \| u \|_{ L_{t}^{\infty} \tilde{H}^{k} (J)} \cdot
\end{array}
\label{Eqn:SobolevIneq2}
\end{equation}
If $u$ is a solution of $i \partial_{t} u + \triangle u = G$, $u(t=0):=u_{0}$ on $J$ with $0 \in J$ and with data $u_0 \in H^{k}$ then the Strichartz
estimates (see for example \cite{keeltao}) yield

\begin{equation}
\begin{array}{l}
\| u \|_{L_{t}^{\infty} \dot{H}^{j} (J) } + \| D^{j} u \|_{L_{t}^{\frac{2(n+2)}{n}} L_{x}^{\frac{2(n+2)}{n}} (J)} + \| D^{j} u
\|_{L_{t}^{\frac{2(n+2)}{n-2}} L_{x}^{\frac{2 n(n+2)}{n^{2}+4}} (J)} \\
\lesssim \| D^{j} G \|_{L_{t}^{\frac{2(n+2)}{n+4}} L_{x}^{\frac{2(n+2)}{n+4}} ( J ) } + \| u_{0} \|_{\dot{H}^{j}}
\end{array}
\label{Eqn:Strich}
\end{equation}
if $j \in \{1,k \}$; if $t_{0} \in J$ then we write

\begin{equation}
\begin{array}{ll}
u(t)= u_{l,t_{0}}(t) + u_{nl,t_{0}}(t)
\end{array}
\end{equation}
with $u_{l,t_{0}}$ denoting the linear part starting from $t_{0}$, i.e

\begin{equation}
\begin{array}{ll}
u_{l,t_{0}}(t) & : = e^{i(t-t_{0}) \triangle} u(t_{0})
\end{array}
\end{equation}
and $u_{nl,t_{0}}$ denoting the nonlinear part starting from $t_{0}$, i.e

\begin{equation}
\begin{array}{ll}
u_{nl,t_{0}}(t) & := - i \int_{t_{0}}^{t} e^{i(t-s) \triangle} G(s) \,ds \cdot
\end{array}
\end{equation}
If $u$ is a $\tilde{H}^{k}-$ solution of (\ref{Eqn:BarelySchrod}) on $J$ with $k > \frac{n}{2}$ \footnote{Let $I$ be an interval. In the sequel we say that $u$ is an $\tilde{H}^{k}-$ solution on $I$ if for all $(t_0,t) \in I^{2}$ such that $t_0 \leq t$ (\ref{Eqn:ClassSol}) and (\ref{Eqn:DistribSchrodg}) hold with $[0,T_l]$ (resp. the interval of integration) replaced with $I$ (resp. $[t_0,t]$).}  and if $t \in J$, then it has a finite energy \footnote{Hence
the norm $\tilde{H}^{k}$  controls the energy. In other words, it is above the energy norm. }

\begin{equation}
\begin{array}{ll}
E(u(t))  & := \frac{1}{2} \int_{\mathbb{R}^{n}} |\nabla u (t,x)|^{2} +  \int_{\mathbb{R}^{n}} F(u,\bar{u})(t,x) \, dx
\end{array}
\label{Eqn:EnergyBarely}
\end{equation}
with

\begin{equation}
\begin{array}{ll}
F(z,\bar{z}) & := \int_{0}^{|z|} t^{\frac{n+2}{n-2}} g(t) \, dt
\end{array}
\end{equation}
Indeed

\begin{equation}
\begin{array}{ll}
\left| \int_{\mathbb{R}^{n}} F(u,\bar{u})(t,x) \, dx \right| & \lesssim  \| u(t) \|^{\frac{2n}{n-2}}_{L^{\frac{2n}{n-2}}} g( \| u(t)
\|_{L^{\infty}} ) \\
&  \lesssim \| u(t) \|^{\frac{2n}{n-2}}_{\dot{H}^{1}} g(\| u(t) \|_{\tilde{H}^{k}}):
\end{array}
\end{equation}
 this follows from a simple integration by part

\begin{equation}
\begin{array}{ll}
F(z,\bar{z}) & \sim  |z|^{\frac{2n}{n-2}} g(|z|)
\end{array}
\label{Eqn:EquivF}
\end{equation}
combined with (\ref{Eqn:SobolevIneq2}). A simple computation shows that the energy is conserved, or, in other words, that
$E(u(t))=E(u_{0})=E$ \footnote{More precisely, the computation holds for smooth solution (i.e solutions in $\tilde{H}^{p}$ with exponents $p$ large enough).
Then $E(u(t))=E(u_{0})$ holds for an $\tilde{H}^{k}-$ solution by a standard approximation argument with smooth solutions.}. Let $\chi$ be a smooth, radial function supported on $|x| \leq 2$ such that $\chi(x)=1$ if $|x| \leq 1$. If $x_{0} \in \mathbb{R}^{n}$, $R>0$ and $u$ is an $\tilde{H}^{k}$ solution of (\ref{Eqn:BarelySchrod}) then we define the mass within the ball $B(x_{0},R)$

\begin{equation}
\begin{array}{ll}
Mass \left( B(x_{0},R),u(t)  \right) & :=  \left( \int_{B(x_{0},R)} |u (t,x) |^{2} \, dx \right)^{\frac{1}{2}}
\end{array}
\end{equation}
Recall (see \cite{taorad}) that \footnote{(\ref{Eqn:UpBdDerivM}) also holds if $u$ is a solution of the linear Schr\"odinger equation
with data in $\tilde{H}^{k}$. }
*
\begin{equation}
\begin{array}{ll}
Mass \left( B(x_{0},R), u(t)  \right) & \lesssim  R \,\sup_{t^{'} \in [0,t]} \| \nabla u(t^{'}) \|_{L^{2}}
\end{array}
\label{Eqn:MassControl}
\end{equation}
and that its derivative satisfies

\begin{equation}
\begin{array}{ll}
|\partial_{t} Mass(u(t),B(x_{0},R))| & \lesssim \frac{\sup_{t^{'} \in [0,t]} \| \nabla u(t^{'}) \|_{L^{2}}}{R}
\end{array}
\label{Eqn:UpBdDerivM}
\end{equation}

Now we set up some notation. We write $a \ll b$ if $a \leq \frac{1}{100} b$, $a \gg b$ is $a \geq 100b $ and $a \sim b $ if $ \frac{1}{100}b \leq a
\leq 100b$, $a \ll_{E} b$ if $a \leq \frac{1}{100 \max{(1,E)}^{100n}} b$ (Here $n$ is the dimension of the space), $a \gg_{E} b$ if $a \geq 100
\max{(1,E)}^{100n} b$, $a \lesssim_{E} b$ if $ a \leq  100 (\max({1,E}))^{100n} b $ and $a \sim_{E} b$ if $ \frac{1}{100}{\max{(1,E)}^{100n}} b
\leq a \leq 100 \max{(1,E)}^{100n} b $. We say that $\tilde{C}$ is the constant determined by $a \lesssim b$ (or $a \lesssim_{E} b$) if it is
the smallest constant $C$ (or $C=C(E)$) that satisfies $a \leq C b$. If $u$ is a function then $u_{h}$ is the function defined by $x \rightarrow
u_{h}(x):=u(x-h)$. If $x \in \mathbb{R} $ then $x+ = x + \epsilon$ for $ 0 < \epsilon \ll 1$. Let $j \in \mathbb{N}$. If $J$ is an interval then we define

\begin{equation}
\begin{array}{ll}
Q_{j}(J,u):=  \| u \|_{L_{t}^{\infty} \tilde{H}^{j}(J)} + \| D u \|_{L_{t}^{\frac{2(n+2)}{n}} L_{x}^{\frac{2(n+2)}{n}} (J)} + \| D^{j} u
\|_{L_{t}^{\frac{2(n+2)}{n}} L_{x}^{\frac{2(n+2)}{n}} (J) } + \| u \|_{L_{t}^{\frac{2(n+2)}{n-2}} L_{x}^{\frac{2(n+2)}{n-2}} (J)}
\end{array}
\end{equation}

\begin{rem}
If $j=k$ then  we write $Q(J,u)$ instead of $Q_{j}(J,u)$.
\end{rem}
If $X$ is a normed vector space endowed with the norm $\| . \|_{X}$ and $R >0$, then $\mathcal{B}(X,R):= \left\{ y \in X, \| y \|_{X} \leq R \right\} $. \\
\\
Now we explain how this paper is organized. In Section \ref{Sec:Thmmain} we prove the main result of this paper, i.e Theorem \ref{thm:main}. The
proof relies upon the following bound of $\| u \|_{L_{t}^{\frac{2(n+2)}{n-2}} L_{x}^{\frac{2(n+2)}{n-2}} }$ on an arbitrarily long time interval

\begin{prop}{\textbf{`` Bound of $L_{t}^{\frac{2(n+2)}{2(n-2)}} L_{x}^{\frac{2(n+2)}{2(n-2)}}$ norm ''}}
Let $u$ be a radial $\tilde{H}^{k} -$ solution of (\ref{Eqn:BarelySchrod}) on a compact interval $J$. There exist three constants $C_{1} \gg_{E} 1$, $C_{2}
\gg_{E} 1$, and $a_{n}>0$ such that if $\| u \|_{L_{t}^{\infty} \tilde{H}^{k}(J)} \leq M$ for some $M \gg 1$, then

\begin{equation}
\begin{array}{ll}
\| u \|^{\frac{2(n+2)}{n-2}}_{ L_{t}^{\frac{2(n+2)}{n-2}} L_{x}^{\frac{2(n+2)}{n-2}} (J)} & \leq  \left( C_{1} g^{a_{n}}(M) \right)^{C_{2}
g^{b_{n}+}(M)}
\end{array}
\label{Eqn:BoundLong}
\end{equation}
with $b_{n}$ such that

\begin{equation}
\begin{array}{l}
b_{n} := \left\{
\begin{array}{l}
5772 , \, \,   n=3 \\
\frac{8024}{3} , \, n=4
\end{array}
\right.
\end{array}
\end{equation}
\label{prop:BoundLong}
\end{prop}
By combining this bound with the Strichartz estimates, we can prove, by induction, that in fact this norm and other norms  (such as $\| u
\|_{L_{t}^{\infty} \tilde{H}^{k}(J)}$, $ \| D u \|_{L_{t}^{\frac{2(n+2)}{n}} L_{x}^{\frac{2(n+2)}{n}} (J) } $ , etc.) can be bounded only by a
constant only depending on the norm of the initial data for $\frac{n}{2} < k  < \frac{2+n}{n-2}$. This already shows (by Proposition \ref{Prop:GlobWellPosedCrit}) global well-posedness of the $\tilde{H}^{k}$-solutions of (\ref{Eqn:BarelySchrod}) for this range of $k$s. In fact we show that that these bounds imply a linear asymptotic behaviour of the solutions, or, in other words, scattering. Proposition \ref{Prop:PersReg}
allows to prove global well-posedness and scattering of solutions of (\ref{Eqn:BarelySchrod}) for the full range, i.e $k > \frac{n}{2}$. The rest of the paper is devoted to prove Proposition \ref{prop:BoundLong}. First we prove a weighted Morawetz-type estimate: it shows, roughly speaking, that the $L_{t}^{\frac{2n}{n-2}} L_{x}^{\frac{2n}{n-2}}$ norm of the solution cannot
concentrate around the origin on long time intervals. Then we modify arguments from Bourgain \cite{bour}, Grillakis \cite{grill} and mostly Tao
\cite{taorad}. We divide $J$ into subintervals $(J_{l})_{ 1 \leq l \leq L}$ such that the $ L_{t}^{\frac{2(n+2)}{n-2}}
L_{x}^{\frac{2(n+2)}{n-2}} $ norm of $u$ is small but also substantial. We prove that, on most of these intervals, the mass on at least one ball
concentrates. By using the radial assumption, we prove that in fact the mass on a ball centered at the origin concentrates. This implies, by
using the Morawetz-type estimate that there exists a significant number  of intervals (in comparison with $L$) that concentrate around a point
$\bar{t}$ and such that the mass concentrates around the origin. But, by H\"older, this implies that $L$ is finite: if not it would violate the
fact that the $L_{t}^{\infty} L_{x}^{\frac{2n}{n-2}} $ norm of the solution is bounded by some power of the energy. The process involves several
tuning parameters. The fact that these parameters depend on the energy is not important; however, it is crucial to understand how they depend on
$g(M)$ since this will play a prominent role in the choice of $c_{n}$ for which we have global well-posedness and scattering of $ \tilde{H}^{k}
$-solutions of (\ref{Eqn:BarelySchrod}) ( with $g(|u|):= \log^{c} \left( \log{(10 + |u|^{2})}  \right)$ and $c< c_{n}$): see the proof of
Theorem \ref{thm:main}, Section \ref{Sec:Thmmain}.

\section{Local well-posedness and criterion for global well-posedness}
\label{Sec:LocalWell}

In this section we prove Proposition \ref{Prop:LocalWell} and Proposition \ref{Prop:GlobWellPosedCrit}.

\subsection{Proof of Proposition \ref{Prop:LocalWell}}

This is done by a modification of standard arguments to establish a local well-posedness theory for (\ref{Eqn:EnergyCrit}).

We define

\begin{equation}
\begin{array}{ll}
X & := \mathcal{C} ( [0,T_{l}], \tilde{H}^{k} ) \cap L_{t}^{\frac{2(n+2)}{n}}  D^{-1} L_{x}^{\frac{2(n+2)}{n}} ([0,T_{l}]) \cap
L_{t}^{\frac{2(n+2)}{n}}  D^{-k} L_{x}^{\frac{2(n+2)}{n}} ([0,T_{l}]) \cap L_{t}^{\frac{2(n+2)}{n-2}} L_{x}^{\frac{2(n+2)}{n-2}} ([0,T_{l}])
\end{array}
\end{equation}
and, for some $C>0$ to be chosen later,

\begin{equation}
\begin{array}{ll}
X_{1} & := \mathcal{B} \left( \mathcal{C} ( [0,T_{l}], \tilde{H}^{k} )
\cap L_{t}^{\frac{2(n+2)}{n}} D^{-1}  L_{x}^{\frac{2(n+2)}{n}} ([0,T_{l}])
\cap L_{t}^{\frac{2(n+2)}{n}} D^{-k} L_{x}^{\frac{2(n+2)}{n}}([0,T_l]) , 2 C M
\right)
\end{array}
\end{equation}
and

\begin{equation}
\begin{array}{ll}
X_{2} & := \mathcal{B} \left( L_{t}^{\frac{2(n+2)}{n-2}} L_{x}^{\frac{2(n+2)}{n-2}} ([0,T_{l}]), 2 \delta \right)
\end{array}
\end{equation}
$X_{1} \cap X_{2} $ is a closed space of the Banach space $X$: therefore it is also a Banach space.

\begin{equation}
\begin{array}{l}
\Psi : = X_{1} \cap X_{2} \rightarrow X_{1} \cap X_{2} \\
u  \rightarrow \Psi(u):t \rightarrow \Psi(u)(t):=e^{i t \triangle} u_0 - i \int_{0}^{t} e^{i(t-t^{'}) \triangle} \left(  |u|^{\frac{4}{n-2}} (t^{'}) u(t^{'}) g(|u(t')|) \right) \, d t^{'}
\end{array}
\end{equation}

\begin{itemize}

\item $\Psi$ maps $X_{1} \cap X_{2}$ to $X_{1} \cap X_{2}$

By the fractional Leibnitz rule (see Appendix with $F(x):=\log^{c} \log(10+x)$, $G(x,\bar{x}):=|x|^{\frac{4}{n-2}} x$ and $\beta:= \frac{4}{n-2}$ ) and
(\ref{Eqn:SobolevIneq2}) we have

\begin{equation}
\begin{array}{ll}
\left\| D^{j} (|u|^{\frac{4}{n-2}} u g(|u|) \right\|_{L_{t}^{\frac{2(n+2)}{n+4}}  L_{x}^{\frac{2(n+2)}{n+4}} ([0,T_{l}])  } & \lesssim \| D^{j}
u \|_{L_{t}^{\frac{2(n+2)}{n}} L_{x}^{\frac{2(n+2)}{n}} ([0,T_{l}])} \| u \|^{\frac{4}{n-2}}_{L_{t}^{\frac{2(n+2)}{n-2}}
L_{x}^{\frac{2(n+2)}{n-2}} ([0,T_{l}])  } \\
&  g ( \| u \|_{L_{t}^{\infty} \tilde{H}^{k} ([0,T_{l}])  } )
\end{array}
\label{Eqn:EstDj}
\end{equation}
if $j=1$. Now assume that $j=k$. By Proposition \ref{Prop:EstHighReg} we get

\begin{equation}
\begin{array}{ll}
\left\| D^{j} (|u|^{\frac{4}{n-2}} u g(|u|) \right\|_{L_{t}^{\frac{2(n+2)}{n+4}}  L_{x}^{\frac{2(n+2)}{n+4}} ([0,T_{l}])}
& \lesssim \delta^{\frac{4}{n-2}} \langle M \rangle^{\bar{C}} \cdot
\end{array}
\nonumber
\end{equation}
Therefore \footnote{In the sequel we allow $\bar{C}$ to change from one line to the
other one.} by the Strichartz estimates (\ref{Eqn:Strich}) and the Sobolev embedding (\ref{Eqn:SobolevIneq2}) we have

\begin{equation}
\begin{array}{ll}
\| u \|_{L_{t}^{\infty} \tilde{H}^{k} ([0,T_{l}])} + \| D  u \|_{L_{t}^{\frac{2(n+2)}{n}} L_{x}^{\frac{2(n+2)}{n}} ([0,T_{l}])} + \| D^{k} u
\|_{L_{t}^{\frac{2(n+2)}{n}} L_{x}^{\frac{2(n+2)}{n}} ([0,T_{l}])} & \lesssim M + \delta^{\frac{4}{n-2}} \langle M \rangle^{\bar{C}}
\end{array}
\label{Eqn:ExistLoc1}
\end{equation}
Moreover

\begin{equation}
\begin{array}{ll}
\| u \|_{L_{t}^{\frac{2(n+2)}{n-2}} L_{x}^{\frac{2(n+2)}{n-2}} ([0,T_{l}])} -  \| e^{ i t \triangle} u_{0} \|_{L_{t}^{\frac{2(n+2)}{n-2}}
L_{x}^{\frac{2(n+2)}{n-2}} ([0,T_{l}])} & \lesssim  \left\| D (|u|^{\frac{4}{n-2}} u g(|u|)  ) \right\|_{L_{t}^{\frac{2(n+2)}{n+4}}
L_{x}^{\frac{2(n+2)}{n+4}} ([0,T_{l}])} \\
& \lesssim   \delta^{\frac{4}{n-2}} M g(M)
\end{array}
\end{equation}
so that

\begin{equation}
\begin{array}{ll}
\| u \|_{L_{t}^{\frac{2(n+2)}{n-2}} L_{x}^{\frac{2(n+2)}{n-2}} ([0,T_{l}])} -  \delta & \lesssim \delta^{\frac{4}{n-2}} M g(M)
\end{array}
\label{Eqn:ExistLoc2}
\end{equation}
Therefore if let let $C$ be equal to the maximum of the constants determined by (\ref{Eqn:ExistLoc1}) and (\ref{Eqn:ExistLoc2}), then we see
that $\Psi(X_{1} \cap X_{2}) \subset X_{1} \cap X_{2}$, provided that $\delta= \delta(M)> 0$ is small enough.

\item $\Psi$ is a contraction. Indeed, by the fundamental theorem of calculus and Proposition
\ref{Prop:EstHighReg}

\begin{equation}
\begin{array}{l}
\| \Psi(u) - \Psi(v) \|_{X}  \\
\lesssim
\sum \limits_{j=1} \left\| D^{j} (u-v) \right\|_{L_{t}^{\frac{2(n+2)}{n-2}} L_{x}^{\frac{2(n+2)}{n-2}} ([0,T_l])}
\left(
g \left( \| u \|_{L_{t}^{\infty} \tilde{H}^{k} ([0,T_l]}  \right) +
g \left( \| v \|_{L_{t}^{\infty} \tilde{H}^{k}([0,T_l]) }  \right)
\right) \\
\left(
\| u \|^{\frac{4}{n-2}}_{L_{t}^{\frac{2(n+2)}{n-2}} L_{x}^{\frac{2(n+2)}{n-2}} ([0,T_{l}])}
+ \| v \|^{\frac{4}{n-2}}_{L_{t}^{\frac{2(n+2)}{n-2}} L_{x}^{\frac{2(n+2)}{n-2}} ([0,T_{l}])}
\right) \\
\\
+ \left(
\begin{array}{l}
\sum  \limits_{\substack{j \in \{ 1,k \} \\ \delta_1 + \delta_2 = \frac{4}{n-2}}}
\left\| D^{j} \left( w_{\tau}^{\delta_1} \overline{w_{\tau}}^{\delta_2} g(|w_{\tau}|) \right) \right\|_{L_{t}^{\frac{n+2}{3}} L_{x}^{\frac{n+2}{3}} ([0,T_l])} \\
+ \sum \limits_{\substack{j \in \{1,k \} \\ \delta_1 + \delta_2 = \frac{4}{n-2} \\ \delta_3 + \delta_4 = 2 }}
\left\| D^{j} \left( w^{\delta_1} \overline{w_{\tau}}^{\delta_2} w_{\tau}^{\delta_3} \overline{w_{\tau}}^{\delta_4} \tilde{g}^{'} \left( |w_{\tau}|^{2} \right) \right) \right\|_{L_{t}^{\frac{n+2}{3}} L_{x}^{\frac{n+2}{3}} ([0,T_l])}
\end{array}
\right)
\| u -v \|_{L_{t}^{\frac{2(n+2)}{n-2}} L_{x}^{\frac{2(n+2)}{n-2}} ([0,T_l])} \\
\lesssim \delta \langle M \rangle^{\bar{C}} \| u -v \|_{X}
\end{array}
\end{equation}

and if $\delta= \delta(M) >0$ is small enough then $\Psi$ is a contraction.

\end{itemize}

\subsection{Proof of Proposition \ref{Prop:GlobWellPosedCrit}}
Again, this is done by a modification of standard arguments used to prove a criterion of global well-posedness of (\ref{Eqn:Schrodpowerp}) (See
\cite{triroybar} for similar arguments). Let $ \frac{n}{2} <  k' < \min \left( k, \frac{n+2}{n-2} \right)$.  Assume that $ \| u \|_{L_{t}^{\frac{2(n+2)}{n-2}} L_{x}^{\frac{2(n+2)}{n-2}} (I_{max})} < \infty$. Then

\begin{itemize}

\item First step: $Q_{k'}(I_{max},u) < \infty$. Indeed, let $0< \epsilon \ll 1$. Let $C$ be the constant determined by $\lesssim$ in (\ref{Eqn:Strich}).
We may assume without loss of generality that $C \gg \max \left( \| u_0 \|^{100}_{\tilde{H}^{k}}, \frac{1}{\| u_0 \|^{100}_{\tilde{H}^{k}}} \right)$.
We divide $I_{max} \cap [0,\infty)$ into subintervals $(I_{j})_{1 \leq j \leq J}$ such that $0 \in I_{1}$,

\begin{equation}
\begin{array}{ll}
\| u \|_{L_{t}^{\frac{2(n+2)}{n-2}} L_{x}^{\frac{2(n+2)}{n-2}} (I_{j})} & = \frac{\epsilon}{g^{\frac{n-2}{4}}
\left( (2C)^{j} \| u_0 \|_{\tilde{H}^{k}} \right) }
\end{array}
\end{equation}
if $1 \leq j <J$ and

\begin{equation}
\begin{array}{ll}
\| u \|_{L_{t}^{\frac{2(n+2)}{n-2}} L_{x}^{\frac{2(n+2)}{n-2}} (I_{J})} & \leq \frac{\epsilon}{g^{\frac{n-2}{4}}
\left( (2C)^{J} \| u_0 \|_{\tilde{H}^{k}} \right)}
\end{array}
\end{equation}
Notice that such a partition always exists since, for $J$ large enough,

\begin{equation}
\begin{array}{ll}
\sum \limits_{j=1}^{J-1} \frac{\epsilon^{\frac{2(n+2)}{n-2}}}{g^{\frac{n+2}{2}} \left( (2C)^{j} \| u_{0} \|_{\tilde{H}^{k}} \right) } & \gtrsim
\sum \limits_{j=1}^{J-1} \frac{1}{\log{( (2C)^{j} \| u_{0} \|_{\tilde{H}^{k}}) } } \\
& = \sum \limits_{j=1}^{J-1} \frac{1}{j \log{(2C)} + \log {(\| u_{0} \|_{\tilde{H}^{k}})} } \\
& \geq \| u \|^{\frac{2(n+2)}{n-2}}_{L_{t}^{\frac{2(n+2)}{n-2}} L_{x}^{\frac{2(n+2)}{n-2}} (I_{max}) }
\end{array}
\end{equation}
By the fractional Leibnitz rule (see Appendix $A$) and (\ref{Eqn:Strich}) we have for some positive constant $C'$

\begin{equation}
\begin{array}{ll}
Q_{k'}(I_{1},u) & \leq C \| u_{0} \|_{\tilde{H}^{k'}} + C  \| D ( |u|^{\frac{4}{n-2}} u g(|u|) ) \|_{ L_{t}^{\frac{2(n+2)}{n+4}}
L_{x}^{\frac{2(n+2)}{n+4}} (I_{1}) } + C \| D^{k'} ( |u|^{\frac{4}{n-2}} u g(|u|) ) \|_{ L_{t}^{\frac{2(n+2)}{n+4}}
L_{x}^{\frac{2(n+2)}{n+4}} (I_{1}) }  \\
& \leq C  \| u_{0} \|_{\tilde{H}^{k'}} + C^{'} C (  \| D u \|_{L_{t}^{\frac{2(n+2)}{n}} L_{x}^{\frac{2(n+2)}{n}}(I_{1})} + \| D^{k'} u
\|_{L_{t}^{\frac{2(n+2)}{n}} L_{x}^{\frac{2(n+2)}{n}}(I_{1})} ) \| u \|^{\frac{4}{n-2}}_{L_{t}^{\frac{2(n+2)}{n-2}} L_{x}^{\frac{2(n+2)}{n-2}}
(I_{1})} \\
& g(\| u \|_{L_{t}^{\infty} \tilde{H}^{k'}(I_{1})}) \\
& \leq C \| u_{0} \|_{\tilde{H}^{k'}} + 2 C^{'} C Q_{k'}(I_{1},u) \| u \|^{\frac{4}{n-2}}_{L_{t}^{\frac{2(n+2)}{n-2}} L_{x}^{\frac{2(n+2)}{n-2}} (I_{1}) }
g(Q_{k'}(I_{1},u))
\end{array}
\label{Eqn:BdQkpr}
\end{equation}
and by a continuity argument\footnote{Let $K \subset I_1$. Then (\ref{Eqn:BdQkpr}) also holds if
$I_1$ is replaced with $K$.}, $ Q_{k'}(I_{1},u) \leq 2 C \| u_{0} \|_{\tilde{H}^{k'}}$ . By iteration $Q_{k'}(I_{j},u) \leq (2C)^{j} \| u_{0}
\|_{\tilde{H}^{k'}}$. Therefore $Q_{k'}(I_{max},u) < \infty$, proceeding similarly on $I_{max} \cap (-\infty,0]$.

\item Second step. We write $I_{max}=(a_{max},b_{max})$. Choose $\bar{t} < b_{max}$ close enough to $b_{max}$ so that
$\| D u \|_{L_{t}^{\frac{2(n+2)}{n}} L_{x}^{\frac{2(n+2)}{n}}([\bar{t}, b_{max}))}  \ll \delta$ and
$ \| u \|_{L_{t}^{\frac{2(n+2)}{n-2}} L_{x}^{\frac{2(n+2)}{n-2}}([\bar{t},b_{max}])} \ll \delta$, with
$\delta$ defined in Proposition \ref{Prop:LocalWell}. We have

\begin{equation}
\begin{array}{ll}
\| e^{i (t -\bar{t}) \triangle} u(\bar{t}) \|_{L_{t}^{\frac{2(n+2)}{n-2}} L_{x}^{\frac{2(n+2)}{n-2}}([\bar{t},b_{max}))} & \leq
\| u \|_{L_{t}^{\frac{2(n+2)}{n-2}} L_{x}^{\frac{2(n+2)}{n-2}}([\bar{t},b_{max}))} + \\
& C' C \| D  u \|_{L_{t}^{\frac{2(n+2)}{n}} L_{x}^{\frac{2(n+2)}{n}}([\bar{t},b_{max}))}
\| u \|^{\frac{4}{n-2}}_{L_{t}^{\frac{2(n+2)}{n-2}} L_{x}^{\frac{2(n+2)}{n-2}}([\bar{t},b_{max}))} \\
&  g( \| u \|_{L_{t}^{\infty} \tilde{H}^{k'}(I_{max})}) \\
& \\
& \leq \frac{3 \delta}{4} \cdot
\end{array}
\end{equation}
Also observe that $\| e^{i (t -\bar{t}) \triangle} u(\bar{t}) \|_{L_{t}^{\frac{2(n+2)}{n-2}} L_{x}^{\frac{2(n+2)}{n-2}}([\bar{t},b_{max}))}
\lesssim \| u(\bar{t}) \|_{\dot{H}^{1}} < \infty$. Hence by the monotone convergence theorem, there exists $\epsilon > 0$ such that \\
$ \| e^{i (t - \bar{t}) \triangle} u(\bar{t}) \|_{L_{t}^{\frac{2(n+2)}{n-2}} L_{x}^{\frac{2(n+2)}{n-2}} [ \tilde{t}, b_{max} + \epsilon]} \leq \delta $. Hence contradiction with Proposition \ref{Prop:LocalWell}.
\end{itemize}

\section{Proof of Theorem \ref{thm:main}}
\label{Sec:Thmmain}

The proof makes use of Proposition \ref{prop:BoundLong} and is made of two steps:

\begin{itemize}

\item finite bound of  $ \|  u \|_{L_{t}^{\infty} \tilde{H}^{k}(\mathbb{R}) }$, $\| u \|_{L_{t}^{\frac{2(n+2)}{n-2}}
L_{x}^{\frac{2(n+2)}{n-2}}(\mathbb{R})}$, $ \| D u \|_{L_{t}^{\frac{2(n+2)}{n}} L_{x}^{\frac{2(n+2)}{n}} (\mathbb{R}) }$ and $ \| D^{k} u
\|_{L_{t}^{\frac{2(n+2)}{n}} L_{x}^{\frac{2(n+2)}{n}} (\mathbb{R})  }$ for $ \frac{n}{2} < k < \frac{n+2}{n-2}$. By time reversal symmetry \footnote{i.e if $ t \rightarrow u(t,x)$ is a solution of (\ref{Eqn:BarelySchrod}) then $t \rightarrow \bar{u}(-t,x)$ is also a solution of (\ref{Eqn:BarelySchrod}).} and by monotone
convergence it is enough to find, for all $T \geq 0$, a finite bound of all these norms restricted to $[0,T]$ and the bound should not depend on
$T$. We define

\begin{equation}
\begin{array}{ll}
\mathcal{F} & : = \left\{ T \in [0, \infty): \sup_{t \in [0,T]} Q([0,t],u)  \leq M_{0}  \right \}
\end{array}
\end{equation}
We claim that $\mathcal{F}= [0,\infty)$ for $M_{0}$, a large constant (to be chosen later)  depending only on $ \| u_{0} \|_{\tilde{H}^{k}}$.
Indeed

\begin{itemize}

\item $0 \in \mathcal{F}$.

\item $\mathcal{F}$ is closed by continuity

\item $\mathcal{F}$ is open. Indeed let $T \in \mathcal{F}$. Then, by continuity there exists $\delta > 0$ such that for $T^{'} \in [0, T + \delta]$
we have  $ Q ( [0,T^{'}])  \leq 2 M_{0} $. In view of (\ref{Eqn:BoundLong}), this implies, in particular, that

\begin{equation}
\begin{array}{ll}
\| u \|^{\frac{2(n+2)}{n-2}}_{L_{t}^{\frac{2(n+2)}{n-2}} L_{x}^{\frac{2(n+2)}{n-2}} ([0,T^{'}])} & \leq C_{1}  \left( g^{a_{n}} ( 2 M_{0} )
\right)^{C_{2} g^{b_{n}+}(2M_{0})}
\end{array}
\label{Eqn:ControlApCrit}
\end{equation}
Let $J:=[0,a]$ be an interval. We get from (\ref{Eqn:Strich}), Proposition \ref{Prop:FracLeibn}, and the Sobolev inequality $\| u \|_{L_{t}^{\infty} L_{x}^{\infty} (J) } \lesssim \| u \|_{L_{t}^{\infty} \tilde{H}^{k} (J)}  $

\begin{equation}
\begin{array}{l}
Q(J,u) \lesssim \| u_{0} \|_{\tilde{H}^{k}} + \left(  \| D u \|_{L_{t}^{\frac{2(n+2)}{n}} L_{x}^{\frac{2(n+2)}{n}} (J) } + \| D^{k} u
\|_{L_{t}^{\frac{2(n+2)}{n}} L_{x}^{\frac{2(n+2)}{n}} (J) } \right) \| u \|
^{\frac{4}{n-2}}_{L_{t}^{\frac{2(n+2)}{n-2}} L_{x}^{\frac{2(n+2)}{n-2}} (J) } \\
 g \left( \| u \|_{L_{t}^{\infty} \tilde{H}^{k} (J)} \right) \\
\lesssim \| u_{0} \|_{\tilde{H}^{k}} + Q(J,u) \| u \| ^{\frac{4}{n-2}}_{L_{t}^{\frac{2(n+2)}{n-2}} L_{x}^{\frac{2(n+2)}{n-2}} (J)} g
\left( Q (J,u) \right)
\end{array}
\label{Eqn:ModelStrichartz}
\end{equation}
Let $C$ be the constant determined by $\lesssim$ in (\ref{Eqn:ModelStrichartz}). We may assume without loss of generality that $C \gg \max{
\left( \| u_{0} \|^{100}_{\tilde{H}^{k}}, \frac{1} {\| u_{0} \|^{100}_{\tilde{H}^{k}}} \right)} $. Let $ 0 <\epsilon \ll 1$. Notice that if $J$
satisfies $ \| u \|_{L_{t}^{\frac{2(n+2)}{n-2}} L_{x}^{\frac{2(n+2)}{n-2}} (J)} = \frac{\epsilon}{g^{\frac{n-2}{4}} ( 2C \| u_{0}
\|_{\tilde{H}^{k}} ) } $ then a simple continuity argument shows that

\begin{equation}
\begin{array}{ll}
Q(J,u)   & \leq 2 C \| u_{0} \|_{\tilde{H}^{k}}
\end{array}
\label{Eqn:BoundQFstIter}
\end{equation}
We divide $[0,T^{'}]$ into subintervals $(J_{i})_{1 \leq i \leq I}$ such that $ \| u \|_{L_{t}^{\frac{2(n+2)}{n-2}} L_{x}^{\frac{2(n+2)}{n-2}}
(J_{i})} = \frac{\epsilon}{g^{\frac{n-2}{4}} ( (2C)^{i} \| u_{0} \|_{\tilde{H}^{k}} ) } $, $ 1 \leq i < I$ and  $ \| u
\|_{L_{t}^{\frac{2(n+2)}{n-2}} L_{x}^{\frac{2(n+2)}{n-2} } (J_{I})} \leq \frac{\epsilon}{g^{\frac{n-2}{4}} ( (2C)^{I} \| u_{0}
\|_{\tilde{H}^{k}} ) } $. Notice that such a partition exists by (\ref{Eqn:ControlApCrit}), the definition of $g$ \footnote{Recall that
$g(x):= \log^{c} \log (10 + x^{2})$} and the following inequality

\begin{equation}
\begin{array}{ll}
\left( C_{1} g^{a_{n}} (2 M_{0}) \right)^{C_{2} g^{b_{n}+} (2M_{0})} & \gtrsim \sum \limits_{i=1}^{I-1} \frac{1}{g^{{\frac{n+2}{2}}} \left( (2C)^{i} \| u_{0} \|_{\tilde{H}^{k}} \right)} \\
& \geq \sum \limits_{i=1}^{I-1} \frac{1} {
\log^{\frac{(n+2)c}{2}} {( \log{( 10 + (2C)^{2i} \| u_{0} \|^{2}_{\tilde{H}^{k}} )} )}} \\
& \gtrsim  \sum \limits_{i=1}^{I-1} \frac{1} { \log^{\frac{(n+2)c}{2}} ( 2i \log{(2C)} + 2 \log{ ( \| u_{0} \|_{\tilde{H}^{k}} ) } )  } \\
& \gtrsim_{\| u_{0} \|_{\tilde{H}^{k}} } \sum \limits_{i=1}^{I-1} \frac{1}{ i^{\frac{1}{2}}} \\
& \gtrsim_{\| u_{0} \|_{\tilde{H}^{k}} } I^{\frac{1}{2}} \\
\end{array}
\label{Eqn:EstI}
\end{equation}
Moreover, by iterating the procedure in (\ref{Eqn:ModelStrichartz}) and (\ref{Eqn:BoundQFstIter}) we get

\begin{equation}
\begin{array}{ll}
Q([0,T^{'}],u) & \lesssim (2C)^{I} \| u_{0} \|_{\tilde{H}^{k}}
\end{array}
\label{Eqn:EstQPr}
\end{equation}
Therefore by (\ref{Eqn:EstI}) there exists $C^{'}=C^{'}(\| u_{0} \|_{\tilde{H}^{k}})$

\begin{equation}
\begin{array}{l}
\log{I}  \lesssim \log{ (C^{'})}  + C_{2} \log^{(b_{n}+) c}  \left( \log{(10 + 4 M^{2}_{0})}   \right) \log{ \left(  C_{1} \log^{a_{n} c}{ (
\log{(10 + 4 M_{0}^{2})})}   \right)}
\end{array}
\end{equation}
and for $ M_{0}= M_{0} (\| u_{0} \|_{\tilde{H}^{k}})  $ large enough

\begin{equation}
\begin{array}{l}
\log{ (C^{'})}  + C_{2} \log^{(b_{n} +) c}  \left( \log{(10 + 4 M^{2}_{0})}   \right)
\log{ \left(  C_{1} \log^{a_{n} c} { ( \log{(10 + 4 M_{0}^{2})})} \right)} \\
\ll \log { \left( \frac{  \log{ \left( \frac{M_{0}}{ \| u_{0} \|_{\tilde{H}^{k}}} \right) }} {\log{(2C)}} \right)}
\end{array}
\end{equation}
since (recall that $c < \frac{1}{b_{n}}$)

\begin{equation}
\begin{array}{ll}
\frac{ \log{ (C^{'})}  + C_{2} \log^{(b_{n}+) c}  \left( \log{(10 + 4 M^{2}_{0})}   \right) \log{ \left(  C_{1} \log^{a_{n} c}{ ( \log{(10 + 4
M_{0}^{2})})}   \right) }} {\log { \left( \frac{  \log{ \left( \frac{M_{0}}{ \| u_{0} \|_{\tilde{H}^{k}}} \right) }} {\log{(2C)}} \right)}} &
\rightarrow_{M_{0} \rightarrow \infty} 0 \cdot
\end{array}
\end{equation}

\end{itemize}

\item Finite bound of $Q(\mathbb{R},u)$ for all $k > \frac{n}{2}$: this follows from Proposition \ref{Prop:PersReg}.

\item Scattering: it is enough to prove that $e^{- i t \triangle} u(t)$ has a limit as $t \rightarrow \infty$ in $\tilde{H}^{k}$. If $t_{1}< t_{2}$
then by dualizing (\ref{Eqn:Strich}) with $G=0$ (more precisely the estimate
$\| D^{j} u \|_{L_{t}^{\frac{2(n+2)}{n}} L_{x}^{\frac{2(n+2)}{n}} ([t_1,t_2])} \lesssim \| u_0 \|_{\dot{H}^{j}}$ ) we get
from Propositions \ref{Prop:FracLeibn} and  \ref{Prop:EstHighReg}

\begin{equation}
\begin{array}{l}
\| e^{-i t_{1} \triangle} u(t_{1}) - e^{- i t_{2} \triangle} u(t_{2}) \|_{\tilde{H}^{k}} \\
 \lesssim \| D^{k} \left( |u|^{\frac{4}{n-2}} u g(|u|)  \right) \|_{L_{t}^{\frac{2(n+2)}{n+4}} L_{x}^{\frac{2(n+2)}{n+4}} ([t_{1},t_{2}]) } +
\| D \left( |u|^{\frac{4}{n-2}} u g(|u|)  \right) \|_{L_{t}^{\frac{2(n+2)}{n+4}} L_{x}^{\frac{2(n+2)}{n+4}} ([t_{1},t_{2}]) } \\
 \lesssim \| u \|^{\frac{4}{n-2}}_{L_{t}^{\frac{2(n+2)}{n-2}} L_{x}^{\frac{2(n+2)}{n-2}} ([t_{1},t_{2}]) }
\end{array}
\end{equation}
and  we conclude from the previous step that given $\epsilon > 0$ there exists $A(\epsilon)$ such that if $t_{2} \geq t_{1} \geq A(\epsilon)$ then
$ \| e^{-i t_{1} \triangle} u(t_{1}) - e^{- i t_{2} \triangle} u(t_{2}) \|_{\tilde{H}^{k}} \leq \epsilon  $. The Cauchy criterion is
satisfied. Hence scattering.

\end{itemize}

\section{Proof of Proposition \ref{prop:BoundLong}}
The proof relies upon a Morawetz type estimate that we prove in the next subsection:

\begin{lem}{ `` \textbf{Morawetz type estimate}''}
Let $u$ be an $\tilde{H}^{k} -$ solution of (\ref{Eqn:BarelySchrod}) on a compact interval $I$. Let $A > 1$. Then

\begin{equation}
\begin{array}{ll}
\int_{I} \int_{|x| \leq A |I|^{\frac{1}{2}}} \frac{ \tilde{F}(u,\bar{u}) (t,x)}{|x|} dx \, dt & \lesssim E  A  |I|^{\frac{1}{2}}
\end{array}
\label{Eqn:MorawEst}
\end{equation}
with

\begin{equation}
\begin{array}{ll}
\tilde{F}(u,\bar{u})(t,x) & := \int_{0}^{|u|(t,x)} s^{\frac{n+2}{n-2}} \left( \frac{4}{n-2} g(s) + s g^{'}(s) \right) \, ds
\end{array}
\label{Eqn:DeftildeF}
\end{equation}
\label{lem:Morawest}
\end{lem}
We prove now Proposition \ref{prop:BoundLong}, following closely an argument in \cite{taorad}. \\

\fbox{\textbf{Step 1}} \\

We divide the interval $J=[t_{1},t_{2}]$ into subintervals $(J_{l}:=[\bar{t}_{l},\bar{t}_{l+1}])_{1  \leq  l \leq L}$ such that

\begin{equation}
\begin{array}{ll}
\| u \|^{\frac{2(n+2)}{n-2}}_{L_{t}^{\frac{2(n+2)}{n-2}}  L_{x}^{\frac{2(n+2)}{n-2}} (J_{l}) } & = \eta_{1}
\end{array}
\end{equation}

\begin{equation}
\begin{array}{ll}
\| u \|^{\frac{2(n+2)}{n-2}}_{L_{t}^{\frac{2(n+2)}{n-2}}  L_{x}^{\frac{2(n+2)}{n-2}} (J_{L}) } & \leq \eta_{1}
\end{array}
\label{Eqn:LastConc}
\end{equation}
with $0 < c_{1}:= c_{1}(E) \ll_{E} 1$ and $\eta_{1} := \frac{c_{1}(E)}{g^{\frac{2(n+2)}{6-n}} (M) } $. It is enough to find an upper bound of $L$ that would
depend on the energy $E$ and $M$. In view of (\ref{Eqn:BoundLong}) , we may replace WLOG the $``\leq''$ sign with the $``=''$ sign in
(\ref{Eqn:LastConc}).  \\
Notice that the value of this parameter, along with the values of the other parameters $\eta_{2}$, $\eta_{3}$ and $\eta$ are not chosen
randomly: they are the largest ones (modulo the energy) such that all the constraints appearing throughout the proof are satisfied. Indeed, if
we consider for example $\eta_{1}$, we basically want to minimize $ L \eta_{1}$. If we go throughout the proof without assigning any value to
$\eta_{1}$ we realize that basically $L \lesssim  \left( \frac{1}{\eta_{1}} \right)^{\frac{1}{\eta_{1}}}$ and therefore $ L \eta_{1} $ is bounded
by a smaller expression as $\eta_{1}$ grows. \\

\fbox {\textbf{Step 2}} \\

We first prove that some norms on these intervals $J_{l}$ are bounded by a constant that depends on the energy.

\begin{res}
We have
\begin{equation}
\begin{array}{ll}
\| D u \|_{L_{t}^{\frac{2(n+2)}{n}} L_{x}^{\frac{2(n+2)}{n}} (J_{l})} & \lesssim_{E} 1
\end{array}
\end{equation}
\label{res:ControlDu}
\end{res}

\begin{proof}

\begin{equation}
\begin{array}{ll}
\| D u \|_{L_{t}^{\frac{2(n+2)}{n}} L_{x}^{\frac{2(n+2)}{n}} (J_{l})} & \lesssim \| D u (\bar{t}_{l}) \|_{L^{2}} + \| D (|u|^{\frac{4}{n-2}} u g(|u|)
) \|_{L_{t}^{\frac{2(n+2)}{n+4}} L_{x}^{\frac{2(n+2)}{n+4}} (J_{l}) } \\
& \lesssim E^{\frac{1}{2}} + \| u \|^{\frac{4}{n-2}}_{L_{t}^{\frac{2(n+2)}{n-2}} L_{x}^{\frac{2(n+2)}{n-2}} (J_{l}) } \| D u
\|_{L_{t}^{\frac{2(n+2)}{n}} L_{x}^{\frac{2(n+2)}{n}} (J_{l})} g(M)  \\
\end{array}
\label{Eqn:DuBd}
\end{equation}
Therefore, by a continuity argument\footnote{Let $K \subset J_l$. Then (\ref{Eqn:DuBd}) also holds if
$J_l$ is replaced with $K$.}, we conclude that $\| D u \|_{L_{t}^{\frac{2(n+2)}{n}} L_{x}^{\frac{2(n+2)}{n}} (J_{l})} \lesssim_{E} 1$.

\end{proof}

\begin{res}
Let $\tilde{J}:= [ \tilde{t}_1, \tilde{t}_2 ] \subset J$ be such that

\begin{equation}
\begin{array}{l}
\frac{\eta_{1}}{2} \leq \| u \|^{\frac{2(n+2)}{n-2}}_{L_{t}^{\frac{2(n+2)}{n-2}} L_{x}^{\frac{2(n+2)}{n-2}} (\tilde{J})} \leq \eta_{1}
\end{array}
\end{equation}
Then
\begin{equation}
\begin{array}{ll}
\| u_{l,\tilde{t}_{j}} \|^{\frac{2(n+2)}{n-2}}_{ L_{t}^{\frac{2(n+2)}{n-2}}  L_{x}^{\frac{2(n+2)}{n-2}} (\tilde{J})} & \gtrsim \eta_{1}
\end{array}
\label{Eqn:LowerBoundLinRes}
\end{equation}
for $j \in \{1,2 \}$. \label{res:linearpartsubs}
\end{res}

\begin{proof}
By Result \ref{res:ControlDu} we have

\begin{equation}
\begin{array}{ll}
\| u - u_{l,\tilde{t}_{j}} \|_{L_{t}^{\frac{2(n+2)}{n-2}} L_{x}^{\frac{2(n+2)}{n-2}} (\tilde{J})}  & \lesssim \| D( |u|^{\frac{4}{n-2}} u g(|u|) )
\|_{L_{t}^{\frac{2(n+2)}{n+4}} L_{x}^{\frac{2(n+2)}{n+4}} (\tilde{J}) } \\
& \lesssim \| D u \|_{L_{t}^{\frac{2(n+2)}{n}}  L_{x}^{\frac{2(n+2)}{n}} (\tilde{J}) } \|  u \|_{L_{t}^{\frac{2(n+2)}{n-2}}
L_{x}^{\frac{2(n+2)}{n-2}} (\tilde{J})}^{\frac{4}{n-2}} g(M) \\
& \lesssim_{E} \|  u \|_{L_{t}^{\frac{2(n+2)}{n-2}} L_{x}^{\frac{2(n+2)}{n-2}} (\tilde{J})}^{\frac{4}{n-2}} g(M) \\
& \ll \eta_{1}^{\frac{n-2}{2(n+2)}}
\end{array}
\end{equation}
Therefore (\ref{Eqn:LowerBoundLinRes}) holds.

\end{proof}

\fbox{\textbf{Step 3}} \\

We define the notion of exceptional intervals and the notion of unexceptional intervals. Let

\begin{equation}
\begin{array}{l}
\eta_{2}:= \left\{
\begin{array}{l}
c_{2} \left(\eta_1 g^{-1}(M) \right)^{22}, \, n=3   \\
c_{2} \left( \eta_{1}^{35} g^{-28}(M) \right)^{\frac{1}{3}}, \, n=4
\end{array}
\right. \end{array}
\label{Eqn:Dfneta2}
\end{equation}
with $ 0< c_{2} \ll_{E} c_{1}$. An interval $J_{l_{0}} = [\bar{t}_{l_{0}}, \bar{t}_{l_{0}+1}]$ of the partition $(J_{l})_{ 1 \leq l \leq L}$ is exceptional if

\begin{equation}
\begin{array}{ll}
\| u_{l,t_{1}} \|^{\frac{2(n+2)}{n-2}}_{L_{t}^{\frac{2(n+2)}{n-2}} L_{x}^{\frac{2(n+2)}{n-2}} (J_{l_{0}})  } + \| u_{l,t_{2}}
\|^{\frac{2(n+2)}{n-2}}_{L_{t}^{\frac{2(n+2)}{n-2}} L_{x}^{\frac{2(n+2)}{n-2}} (J_{l_{0}}) } & \geq \eta_{2}
\end{array}
\end{equation}
Notice that, in view of the Strichartz estimates (\ref{Eqn:Strich}), it is easy to find an upper bound of the cardinal of the exceptional
intervals:

\begin{equation}
\begin{array}{ll}
\card{\{ J_{l}: \, J_{l} \, \mathrm{exceptional} \}} & \lesssim_{E} \eta_{2}^{-1}
\end{array}
\label{Eqn:BoundCardExcep}
\end{equation}

\fbox{ \textbf{Step 4} } \\

Now we prove that on  each unexceptional subintervals $J_{l}$ there is a ball for which we have a mass concentration.

\begin{res}{``\textbf{Mass Concentration}''}
There exists an $x_{l} \in \mathbb{R}^{n}$, two constants $ 0 <  c \ll_{E} 1 $ and $C \gg_{E} 1$ such that for each unexceptional interval $J_{l}$ and
for $t \in J_{l}$

\begin{itemize}

\item if $n=3$

\begin{equation}
\begin{array}{ll}
Mass \left( u(t), B(x_{l}, C  g^{\frac{13}{3}} (M) |J_{l}|^{\frac{1}{2}} )   \right) & \geq c g^{-
\frac{13}{3}} (M) |J_{l}|^{\frac{1}{2}}
\end{array}
\label{Eqn:ConcentrationMassn3}
\end{equation}

\item if $n=4$

\begin{equation}
\begin{array}{ll}
Mass \left( u(t), B(x_{l}, C  g^{\frac{17}{3}} (M) |J_{l}|^{\frac{1}{2}} )   \right) & \geq c g^{-\frac{17}{3}} (M) |J_{l}|^{\frac{1}{2}}
\end{array}
\label{Eqn:ConcentrationMass}
\end{equation}

\end{itemize}

\label{res:ConcentrationMass}
\end{res}

\begin{proof}

By time translation invariance \footnote{i.e if $u$ is a solution of (\ref{Eqn:BarelySchrod}) and $t_{0} \in \mathbb{R}$ then $(t,x) \rightarrow
u(t-t_{0},x)$ is also a solution of (\ref{Eqn:BarelySchrod}). } we may assume that $\bar{t}_{l}=0$. By using the pigeonhole principle and the
reflection symmetry (if necessary) \footnote{if $u$ is a solution of (\ref{Eqn:BarelySchrod})
 then $(t,x) \rightarrow \bar{u}(-t,x) $ is also a solution of (\ref{Eqn:BarelySchrod}). } we may assume that

\begin{equation}
\begin{array}{ll}
\int_{\frac{|J_{l}|}{2}}^{|J_{l}|} \int_{\mathbb{R}^{n}} | u(t,x) |^{\frac{2(n+2)}{n-2}} \, dx \, dt & \geq \frac{\eta_{1}}{4}
\end{array}
\label{Eqn:CriticalNormLarge}
\end{equation}
By the pigeonhole principle there exists $t_{*} $ such that $[(t_{*} -\eta_{3})|J_{l}|, t_{*} |J_{l}| ] \subset \left[0, \frac{|J_{l}|}{2}
\right] $ (with $\eta_{3} \ll 1$) and

\begin{equation}
\begin{array}{ll}
\int_{(t_{*} - \eta_{3}) |J_{l}|}^{t_{*} | J_{l}|} \int_{\mathbb{R}^{n}}  | u(t,x)|^{\frac{2(n+2)}{n-2}} \, dx \, dt \lesssim \eta_{1} \eta_{3}
\end{array}
\end{equation}

\begin{equation}
\begin{array}{ll}
\int_{\mathbb{R}^{n}} | u_{l,t_{1}} ( (t_{*} -\eta_{3}) |J_{l}|,x ) |^{\frac{2(n+2)}{n-2}} \, dx & \lesssim \frac{\eta_{2}}{|J_{l}|}
\end{array}
\label{Eqn:Ptwiseult1}
\end{equation}
Applying Result \ref{res:linearpartsubs} to (\ref{Eqn:CriticalNormLarge}) we have

\begin{equation}
\begin{array}{ll}
\int_{t_{*} |J_{l}|}^{|J_{l}|} \int_{\mathbb{R}^{n}} | e^{i(t -t_{*}|J_{l}|) \triangle} u(t_{*}|J_{l}|,x) |^{\frac{2(n+2)}{n-2}} \, dx \, dt &
\gtrsim_{E} \eta_{1}
\end{array}
\label{Eqn:LowerBoundLin}
\end{equation}
By Duhamel formula we have

\begin{equation}
\begin{array}{ll}
u(t_{*} |J_{l}|) & = e^{i(t_{*}|J_{l}|- t_{1} ) \triangle} u(t_{1}) - i \int_{t_{1}}^{ (t_{*} - \eta_{3}) |J_{l}|  } e^{i(t_{*} |J_{l}|-s)
\triangle} (
|u(s)|^{\frac{4}{n-2}} u(s) g(|u(s)|) ) \, ds \\
& - i \int_{(t_{*} - \eta_{3}) |J_{l}|}^{t_{*}|J_{l}|} e^{i(t_{*} |J_{l}|-s) \triangle} ( |u(s)|^{\frac{4}{n-2}} u(s) g(|u(s)|) ) \, ds
\end{array}
\end{equation}
and, composing this equality with $e^{i(t-t_{*}|J_{l}|) \triangle}$ we get

\begin{equation}
\begin{array}{ll}
e^{i(t-t_{*}|J_{l}|) \triangle} u(t_{*} |J_{l}|) & = u_{l,t_{1}}(t) - i \int_{t_{1}}^{ (t_{*} - \eta_{3}) |J_{l}|  } e^{i(t-s)
\triangle} ( |u(s)|^{\frac{4}{n-2}} u g(|u(s)|) ) \, ds \\
& - i \int_{(t_{*} - \eta_{3}) |J_{l}|}^{t_{*}|J_{l}|} e^{i(t -s) \triangle} ( |u(s)|^{\frac{4}{n-2}} u(s) g(|u(s)|) ) \, ds \\
& = u_{l,t_{1}}(t) + v_{1}(t) + v_{2}(t)
\end{array}
\label{Eqn:Decompulti}
\end{equation}
We get from a variant of the Strichartz estimates (\ref{Eqn:Strich}) and the Sobolev inequality (\ref{Eqn:SobolevIneq1}) \\

\begin{equation}
\begin{array}{l}
\| v_{2} \|_{L_{t}^{\frac{2(n+2)}{n-2}} L_{x}^{\frac{2(n+2)}{n-2}}
\cap L_{t}^{\infty} D^{-1} L_{x}^{2}([t_{*} |J_{l}|, |J_{l}|])} \\
\lesssim \| D(|u|^{\frac{4}{n-2}} u g(|u|))\|_{L_{t}^{\frac{2(n+2)}{n+4}}
L_{x}^{\frac{2(n+2)}{n+4}}([ (t_{*} - \eta_{3}) |J_{l}|, t_{*} |J_{l}| ])} \\
\lesssim \|  D u  \|_{L_{t}^{\frac{2(n+2)}{n}} L_{x}^{\frac{2(n+2)}{n}} ([ (t_{*} -\eta_{3}) |J_{l}|,  t_{*}|J_{l}| ])  } \| u
\|^{\frac{4}{n-2}}_{L_{t}^{\frac{2(n+2)}{n-2}} L_{x}^{\frac{2(n+2)}{n-2}} ([ (t_{*} - \eta_{3}) |J_{l}|, t_{*} |J_{l}| ]) } \\
 g( \| u \|_{L_{t}^{\infty} \tilde{H}^{k}([(t_{*} -\eta_{3}) |J_{l}|, |J_{l}| ]) }  ) \\
 \lesssim_{E} (\eta_{1} \eta_{3})^{\frac{2}{n+2}}  g(M) \\
 \ll \eta_{1}^{\frac{n-2}{2(n+2)}}
\end{array}
\label{Eqn:Dv2Ineq}
\end{equation}
Notice also that $\eta_{2} \ll \eta_{1}$ and that $J_{l}$ is non-exceptional. Therefore $ \| u_{l,t_{1}} \|_{L_{t}^{\frac{2(n+2)}{n-2}}
L_{x}^{\frac{2(n+2)}{n-2}} ([t_{*} |J_{l}|, |J_{l}| ]) } \ll \eta_{1}$ and combining this inequality with (\ref{Eqn:Dv2Ineq}) and
(\ref{Eqn:LowerBoundLin}) we conclude that the  $ L_{t}^{\frac{2(n+2)}{n-2}} L_{x}^{\frac{2(n+2)}{n-2}} $ norm of $v_{1}$ on $[t_{*} |J_{l}|,
|J_{l}| ]$ is bounded from below:

\begin{equation}
\begin{array}{ll}
\| v_{1} \|^{\frac{2(n+2)}{n-2}}_{L_{t}^{\frac{2(n+2)}{n-2}} L_{x}^{\frac{2(n+2)}{n-2}} ([t_{*} |J_{l}|, |J_{l}| ])} & \gtrsim \eta_{1}
\end{array}
\end{equation}
By (\ref{Eqn:Strich}), (\ref{Eqn:Decompulti}) and (\ref{Eqn:Dv2Ineq}) we also have an upper bound of the  $ L_{t}^{\frac{2(n+2)}{n-2}}
L_{x}^{\frac{2(n+2)}{n-2}} $ norm of $v_{1}$ on $[t_{*} |J_{l}|, |J_{l}| ]$

\begin{equation}
\begin{array}{ll}
\| v_{1} \|^{\frac{2(n+2)}{n-2}}_{L_{t}^{\frac{2(n+2)}{n-2}} L_{x}^{\frac{2(n+2)}{n-2}}  ([t_{*} |J_{l}|, |J_{l}|])
\cap L_{t}^{\infty} D^{-1} L_{x}^{2}([t_{*} |J_{l}|, |J_{l}|])  } &  \lesssim_{E} 1
\end{array}
\end{equation}
Now we use a lemma that is proved in Subsection \ref{Subsec:Lemmaregv}.

\begin{lem}{`` \textbf{Regularity of $v_{1}$} ''}
We have

\begin{equation}
\begin{array}{ll}
\| v_{1,h} - v_{1} \|_{L_{t}^{\infty}  L_{x}^{\frac{2(n+2)}{n-2}} ( [t_{*} |J_{l}|, |J_{l}|]) } & \lesssim_{E} |h|^{\alpha} |J_{l}|^{\beta}
\gamma
\end{array}
\label{Eqn:Regv}
\end{equation}
with
\begin{itemize}

\item $\alpha=\frac{1}{5}$ if $n=3$; $\alpha=1$ if $n=4$

\item $\beta=-\frac{1}{5}$ if $n=3$; $\beta = -\frac{2}{3}$ if $n=4$

\item $\gamma= g^{\frac{2}{15}}(M) $ if $n=3$; $\gamma = g^{\frac{1}{3}}(M) $ if $n=4$;

\end{itemize}

\label{lem:regv}
\end{lem}
Denote by $v_{1,h}^{av}(x) := \int \chi(y) v_{1}(x + |h|y) \; dy $ with  $\chi$ a bump function with total mass equal to one and such that
$\supp (\chi) \subset B(0,1)$. Then

\begin{equation}
\begin{array}{ll}
\| v_{1,h}^{av} - v_{1} \|_{L_{t}^{\frac{2(n+2)}{n-2}} L_{x}^{\frac{2(n+2)}{n-2}} ([ t_{*}|J_{l}|,|J_{l}|])  } & \lesssim | J_{l}
|^{\frac{n-2}{2(n+2)}} \| v_{1,h}^{av} - v_{1} \|_{L_{t}^{\infty} L_{x}^{\frac{2(n+2)}{n-2}} ([ t_{*}|J_{l}|,|J_{l}|]) } \\
& \lesssim_{E} |h|^{\alpha} | J_{l}|^{\beta + \frac{n-2}{2(n+2)}} \gamma
\end{array}
\end{equation}
Therefore if $h$ satisfies $ |h| := c_{3} |J_{l}|^{- \frac{ \left( \beta + \frac{n-2}{2(n+2)} \right)}{\alpha}} \gamma^{ - \frac{1}{\alpha}}
\eta^{\frac{n-2}{2(n+2) \alpha}}_{1} $ with $ 0 < c_{3} \ll_{E} 1 $ then

\begin{equation}
\begin{array}{ll}
\| v_{1,h}^{av} \|_{L_{t}^{\frac{2(n+2)}{n-2}} L_{x}^{\frac{2(n+2)}{n-2}} ([ t_{*}|J_{l}|,|J_{l}|]) } & \gtrsim \eta_{1}^{\frac{n-2}{2(n+2)}}
\end{array}
\label{Eqn:Ineqv1hBd}
\end{equation}
Now notice that by the Duhamel formula $v_{1}(t) = u_{l,(t_{*} - \eta_{3})|J_{l}|}(t) - u_{l,t_1} (t)  $ and therefore, by the Strichartz
estimates (\ref{Eqn:Strich}) and the conservation of energy, $\| v_{1} \|_{L_{t}^{\infty} L_{x}^{\frac{2n}{n-2}} ([t_{*} |J_{l}|, |J_{l}|])}
\lesssim_{E} 1 $. From that we get $ \| v_{1,h}^{av} \|_{L_{t}^{\frac{2n}{n-2}} L_{x}^{\frac{2n}{n-2}} ([t_{*} |J_{l}|, |J_{l}|]) } \lesssim_{E}
|J_{l}|^{\frac{n-2}{2n}}$ and, by interpolation,

\begin{equation}
\begin{array}{ll}
\| v_{1,h}^{av} \|_{L_{t}^{\frac{2(n+2)}{n-2}} L_{x}^{\frac{2(n+2}{n-2}}([t_{*} |J_{l}|, |J_{l}|]) } & \lesssim \| v_{1,h}^{av}
\|^{\frac{2}{n+2}}_{L_{t}^{\infty} L_{x}^{\infty}([t_{*} |J_{l}|, |J_{l}|]) } \| v_{1,h}^{av} \|^{\frac{n}{n+2}}_{L_{t}^{\frac{2n}{n-2}}
L_{x}^{\frac{2n}{n-2}} ([t_{*} |J_{l}|, |J_{l}|]) }    \\
\end{array}
\end{equation}
and, in view of (\ref{Eqn:Ineqv1hBd})

\begin{equation}
\begin{array}{ll}
\| v_{1,h}^{av} \|_{L_{t}^{\infty} L_{x}^{\infty} ([ t_{*}|J_{l}|,|J_{l}|])} & \gtrsim |J_{l}|^{- \frac{n-2}{4}} \eta^{\frac{n-2}{4}}_{1}
\end{array}
\label{Eqn:LowerBoundLtinfLxinf}
\end{equation}
Writing $ Mass(v(t),B(x,r))  = r^{\frac{n}{2}}  \left( \int_{|y| \leq 1} | v(t,x + r y)    |^{2} \, dy \right)^{\frac{1}{2}} $ we deduce from
Cauchy Schwartz and (\ref{Eqn:LowerBoundLtinfLxinf}) that there exists  $\check{t}_l \in [t_{*}|J_{l}|, |J_{l}| ]$ and $x_{l} \in \mathbb{R}^{n}$ such
that

\begin{equation}
\begin{array}{ll}
Mass \left( v_{1}(\check{t}_{l}),B(x_{l},|h|) \right) & \gtrsim |J_{l}|^{- \frac{n-2}{4}} \eta_{1}^{\frac{n-2}{4}}  |h|^{\frac{n}{2}}
\end{array}
\end{equation}
Therefore, by (\ref{Eqn:UpBdDerivM}) we see that if $R =C_{3}(E) \eta_{1}^{\frac{2-n}{4}} |J_{l}|^{\frac{2+n}{4}} |h|^{-\frac{n}{2}} $
with $C_3 \gg_E 1$ then

\begin{equation}
\begin{array}{ll}
Mass \left( v_{1}( (t_{*} - \eta_{3}) |J_{l}| ), B(x_{l}, R)   \right) & \gtrsim |J_{l}|^{- \frac{n-2}{4}} \eta^{\frac{n-2}{4}}_{1}
|h|^{\frac{n}{2}}
\end{array}
\end{equation}
Notice that $ u \left( (t_{*} - \eta_{3})|J_{l}|  \right)= u_{l,t_{1}} \left( (t_{*} - \eta_{3})|J_{l}|  \right) - i v_{1} \left( (t_{*} -
\eta_{3}) |J_{l}|  \right) $. By H\"older inequality, (\ref{Eqn:Dfneta2}), and (\ref{Eqn:Ptwiseult1})

\begin{equation}
\begin{array}{ll}
Mass \left( u_{l,t_{1}} ( (t_{*} - \eta_{3}) |J_{l}|), B(x_{l},R) \right) & \lesssim R^{\frac{2n}{n+2}} \frac{\eta^{\frac{n-2}{2(n+2)}}_{2}}
{|J_{l}|^{\frac{n-2}{2(n+2)}}} \\
& \ll |J_{l}|^{-\frac{n-2}{4}} \eta_{1}^{\frac{n-2}{4}} |h|^{\frac{n}{2}}
\end{array}
\end{equation}
Therefore $ Mass \left( u ((t_{*} - \eta_{3}) |J_{l}|) ,B(x_{l},R) \right) \sim  Mass \left( v_{1}( (t_{*} - \eta_{3}) |J_{l}| ), B(x_{l}, R)
\right) $. Applying again (\ref{Eqn:UpBdDerivM}) we get

\begin{equation}
\begin{array}{ll}
Mass \left( u (t), B(x_{l},R) \right) & \gtrsim |J_{l}|^{- \frac{n-2}{4}} \eta^{\frac{n-2}{4}}_{1} |h|^{\frac{n}{2}}
\end{array}
\end{equation}
for $t \in J_{l}$. Putting everything together we get (\ref{Eqn:ConcentrationMassn3}) and (\ref{Eqn:ConcentrationMass}).

\end{proof}

Next we use the radial symmetry to prove that, in fact, there is a mass concentration around the origin. \\

\fbox{ \textbf{Step 5} } \\

\begin{res}{`` \textbf{Mass concentration around the origin} ''}
There exists a positive constant $\ll_{E} 1$ (that we still denote by $c$ to avoid too much notation) and a
constant $\tilde{C} \gg_{E} 1$  such that on each unexceptional interval $J_{l}$ we have

\begin{itemize}

\item if $n=3$

\begin{equation}
\begin{array}{ll}
Mass \left( u(t),  B(0, \tilde{C} g^{\frac{169}{3}}(M) |J_{l}|^{\frac{1}{2}}  )  \right) & \geq c
g^{-\frac{13}{3}}(M) |J_{l}|^{\frac{1}{2}}
\end{array}
\end{equation}

\item if $n=4$

\begin{equation}
\begin{array}{ll}
Mass \left( u(t),  B(0, \tilde{C} g^{51}(M) |J_{l}|^{\frac{1}{2}}  )  \right) & \geq c
g^{-\frac{17}{3}}(M) |J_{l}|^{\frac{1}{2}}
\end{array}
\label{Eqn:ConcMassZero}
\end{equation}

\end{itemize}

\label{res:ConcMassZero}
\end{res}

\begin{proof}

We deal with the case $n=4$. The case $n=3$ is treated similarly and the proof is left to the reader.

Let $A:= \tilde{C} g^{51}(M)$ for some $\tilde{C} \gg_{E} C$ (Recall that $C$ is defined in
(\ref{Eqn:ConcentrationMass}) ). There are (a priori) two options:

\begin{itemize}

\item $ |x_{l}| \geq \frac{A}{2} |J_{l}|^{\frac{1}{2}} $ . Then there are at least $ \frac{A}{ 100 C g^{\frac{17}{3}}
(M) } $ rotations of the ball $ B(x_{l}, C g^{\frac{17}{3}} (M) |J_{l}|^{\frac{1}{2}} ) $ that are disjoint. Now,
since the solution is radial, the mass on each of these balls $B_{j}$ is equal to that of the ball
$ B(x_{l}, C  g^{\frac{17}{3}} (M) |J_{l}|^{\frac{1}{2}}) $. But then by H\"older inequality we have

\begin{equation}
\begin{array}{ll}
\| u(t)\|^{\frac{2n}{n-2}}_{L^{2}(B_{j})} & \leq \| u(t) \|^{\frac{2n}{n-2}}_{L^{\frac{2n}{n-2}}(B_{j})} \left( C g^{\frac{17}{3}} (M) |J_{l}|^{\frac{1}{2}} \right)^{\frac{2n}{n-2}}
\end{array}
\end{equation}
and summing over $j$ we see from the equality $\| u(t) \|^{\frac{2n}{n-2}}_{L^{\frac{2n}{n-2}}} \lesssim E $ that

\begin{equation}
\begin{array}{l}
\frac{A}{ 100 C g^{\frac{17}{3}} (M) } \left( c g^{- \frac{17}{3}}(M)
|J_{l}|^{\frac{1}{2}} \right)^{\frac{2n}{n-2}}  \\
\leq E \left( C g^{\frac{17}{3}} (M) |J_{l}|^{\frac{1}{2}} \right)^{\frac{2n}{n-2}}
\end{array}
\end{equation}
must be true. But with the value of $A$ chosen above we see that this inequality cannot be satisfied if $\tilde{C}$ is large enough. Therefore
this scenario is impossible.

\item $|x_{l}| \leq \frac{A}{2} |J_{l}|^{\frac{1}{2}}  $. Then by
(\ref{Eqn:ConcentrationMass}) and the triangle inequality, we see that (\ref{Eqn:ConcMassZero}) holds.
\end{itemize}

\end{proof}

\begin{rem}
In order to avoid too much notation we will still write in the sequel $C$ for $\tilde{C}$  in (\ref{Eqn:ConcMassZero}).
\end{rem}

\fbox{\textbf{Step 6}} \\

Combining the inequality (\ref{Eqn:ConcMassZero}) to the Morawetz type inequality found in Lemma \ref{lem:Morawest} we can prove that at least
one of the intervals $J_{l}$ is large. More precisely

\begin{res}{``\textbf{One of the intervals $J_{l}$ is large }''}
There exists a positive constant $\ll_{E} 1$ (that we still denote by $c$ to avoid too much notation) and $\tilde{l} \in [ 1,..,L ]$ such that

\begin{itemize}

\item if $n=3$

\begin{equation}
\begin{array}{ll}
|J_{\tilde{l}}|  &  \geq c g^{-\frac{2860}{3}}(M) |J|
\end{array}
\end{equation}

\item if $n=4$

\begin{equation}
\begin{array}{ll}
|J_{\tilde{l}}|  &  \geq c g^{-\frac{1972}{3}}(M) |J|
\end{array}
\label{Eqn:DistribInterv}
\end{equation}

\end{itemize}

\label{res:DistribInterv}
\end{res}

\begin{proof}

Again  we shall treat the case $n=4$. The case $n=3$ is left to the reader.

There are two options:

\begin{itemize}

\item $J_{l}$ is unexceptional. Let $R:=  C g^{51}(M)  |J_{l}|^{\frac{1}{2}} $. By H\"older inequality (in space), by
integration in time we have

\begin{equation}
\begin{array}{ll}
\int_{J_{l}} \int_{B(0,R)}  \frac{|u(t,x)|^{\frac{2n}{n-2}}}{|x|} \, dx dt & \geq  |J_{l}|  Mass^{\frac{2n}{n-2}} \left( u(t), B(0,R) \right)
R^{\frac{2-3n}{n-2}}
\end{array}
\end{equation}
After summation over $l$ we see, by (\ref{Eqn:ConcMassZero})  and  (\ref{Eqn:MorawEst}) that

\begin{equation}
\begin{array}{l}
\sum\limits_{l=1}^{L} |J_{l}|  \left(  g^{-\frac{17}{3}} (M)   |J_{l}|^{\frac{1}{2}} \right)^{\frac{2n}{n-2}} \left( C
g^{51} (M) |J_{l}|^{\frac{1}{2}} \right)^{\frac{2-3n}{n-2}} \\
\lesssim E |J|^{\frac{1}{2}}  g^{51} (M)
\end{array}
\end{equation}
and after rearranging, we see that

\begin{equation}
\begin{array}{ll}
\sum \limits_{l=1}^{L} |J_{l}|^{\frac{1}{2}}   g^{-\frac{986}{3}}(M)   & \lesssim E |J|^{\frac{1}{2}}
\end{array}
\end{equation}

\item $J_{l}$ is exceptional. In this case by (\ref{Eqn:BoundCardExcep}) and

\begin{equation}
\begin{array}{ll}
\sum \limits_{l=1}^{L} |J_{l}|^{\frac{1}{2}} & \lesssim_{E}  \eta_{2}^{-1} \sup_{ 1 \leq
l \leq L} |J_{l}|^{\frac{1}{2}} \\
& \lesssim_{E}  \eta_{2}^{-1} |J|^{\frac{1}{2}}
\end{array}
\end{equation}

\end{itemize}

Therefore, writing $\sum\limits_{l=1}^{L} |J_{l}|^{\frac{1}{2}} \geq \frac{|J|}{\sup_{1 \leq l \leq L} |J_{l}|^{\frac{1}{2}} } $, we conclude that
there exists a constant $\ll_{E} 1$ (still denoted by $c$) and $\tilde{l} \in [1,..,L ]$ such that (\ref{Eqn:DistribInterv}) holds.

\end{proof}

\fbox{\textbf{Step 7}} \\

We use a crucial algorithm due to Bourgain \cite{bour} to prove that there are many of those intervals that concentrate.

\begin{res}{`` \textbf{Concentration of intervals} ''}
Let

\begin{equation}
\begin{array}{l}
\eta:= \left\{
\begin{array}{l}
c g^{-\frac{2860}{3}}(M), \, n=3  \\
c g^{-\frac{1972}{3}} (M), \, n=4
\end{array}
\right.
\end{array}
\end{equation}
There exist a time $\bar{t}$, $K > 0$ and intervals $J_{l_{1}}$, ...., $J_{l_{K}}$ such that

\begin{equation}
\begin{array}{ll}
|J_{l_{1}}| \geq  2 |J_{l_{2}}| ... \geq 2^{k-1} |J_{l_{k}}| ... \geq 2^{K-1} |J_{l_{K}}|,
\end{array}
\end{equation}

\begin{equation}
\begin{array}{ll}
dist(\bar{t}, J_{l_{k}}  ) \leq \eta^{-1} |J_{l_k}|,
\end{array}
\label{Eqn:ControlDist}
\end{equation}
and

\begin{equation}
\begin{array}{ll}
K & \geq - \frac{\log{(L)}}{2 \log { \left(  \frac{\eta}{8} \right)}  } \cdot
\end{array}
\label{Eqn:LowerBoundK}
\end{equation}

\end{res}

\begin{proof}

There are several steps

\begin{enumerate}

\item By Result \ref{res:DistribInterv} there exists an interval $J_{l_{1}}$ such that $|J_{l_{1}}| \geq \eta |J|$. We have
$dist(t,J_{l_{1}}) \leq |J| \leq \eta^{-1} |J_{l_{1}}|$, $t \in J$.
\item Remove all the intervals $J_{l}$ such that $|J_{l}| \geq \frac{|J_{l_{1}}|}{2}$. By the property of $J_{l_{1}}$, there are at
most $ 2 \eta^{-1} $ intervals satisfying this property and consequently there are at most $4 \eta^{-1}$ remaining connected components
resulting from this removal.
\item If $L \leq 100 \eta^{-1}$ then we let $K=1$ and we can check that (\ref{Eqn:LowerBoundK}) is satisfied. If not: one of these connected components
(denoted by $K_{1}$) contains at least $\frac{\eta}{8} L $ intervals. Let $L_{1}$ be the number of intervals making $K_{1}$.
\item Apply (1) again: there exists an interval $J_{l_{2}}$ such that $|J_{l_{2}}| \geq \eta |K_{1}|$ and
$dist(t,J_{l_{2}}) \leq |K_{1}| \leq \eta^{-1} |J_{l_{2}}|$, $t \in K_1$. Apply (2) again: remove all the intervals $J_{l}$ such that $|J_{l}| \geq
\frac{|J_{l_{2}}|}{2}$. By the property of $J_{l_{2}}$, there are at most $ 2 \eta^{-1}$ intervals to be removed and there are at most $4
\eta^{-1}$ remaining connected components. Apply (3) again: if $L_{1} \leq 100 \eta^{-1}$ then  we let $K=2$ and we can check that
(\ref{Eqn:LowerBoundK}) is satisfied, since $K_{1}$ contains at least $ \frac{\eta}{8} L $ intervals; if $L_{1} \geq 100 \eta^{-1}$ then one of
the connected components (denoted by $K_{2}$) contains at least $\frac{\eta}{8} L_{1}$ intervals. Let $L_{2}$ be the number of intervals making
$K_{2}$. Then $L_{2} \geq \left( \frac{\eta}{    8} \right)^{2} L $.
\item We can iterate this procedure $K$ times until $L_{K-1} \leq 100 \eta^{-1} $. It is not difficult to see that $K$ satisfies
(\ref{Eqn:LowerBoundK}), since $L_{K-1} \geq \left( \frac{\eta}{8} \right)^{K-1} L $.
\end{enumerate}
\end{proof}

\fbox{\textbf{Step 8}} \\

We prove that $L< \infty$, by using Step $7$ and the conservation the energy.  More precisely

\begin{res}{\textbf{``finite bound of $L$'' }}
There exist two constants  $C_{1} \gg_{E} 1$ and $C_{2} \gg_{E} 1$  such that

\begin{itemize}

\item if $n=3$

\begin{equation}
\begin{array}{l}
L \leq \left( C_{1} g^{\frac{2860}{3}}(M) \right)^{ C_{2}  g^{5772+} (M)  }
\end{array}
\end{equation}

\item if $n=4$

\begin{equation}
\begin{array}{l}
L \leq \left( C_{1} g^{\frac{1972}{3}}(M) \right)^{ C_{2}  g^{\frac{8024}{3}+} (M)  }
\end{array}
\label{Eqn:BoundLFin}
\end{equation}

\end{itemize}

\end{res}

\begin{proof}
Again we shall prove this result for $n=4$. The case $n=3$ is left to the reader. Let $R_{l_k} := C g^{663}(M) |J_{l_k}|^{\frac{1}{2}}$. By Result \ref{res:ConcentrationMass} we have

\begin{equation}
\begin{array}{ll}
Mass \left( u(t),  B(x_{l_{k}},R_{l_k}) \right) & \geq c g^{- \frac{17}{3}} (M) |J_{l_k}|^{\frac{1}{2}}
\end{array}
\label{Eqn:LowerBondMass2}
\end{equation}
for all $t \in J_{l_{k}}$. Even if it means redefining $C$ \footnote{i.e making it larger than its original value modulo a multiplication by
some power of $\max{(1,E)}$} then we see, by (\ref{Eqn:UpBdDerivM}) and (\ref{Eqn:ControlDist}) that (\ref{Eqn:LowerBondMass2}) holds for
$t=\bar{t}$ with $c$ substituted for $\frac{c}{2}$. On the other hand we see that by (\ref{Eqn:MassControl}) that \footnote{Notation:
$\sum\limits_{k^{'}=k+N}^{K} a_{k^{'}}=0$, if $k+N > K$ }

\begin{equation}
\begin{array}{ll}
\sum\limits_{k^{'}=k+N}^{K} \int_{B(x_{l_{k^{'}}} ,R_{l_k'} )} |u(\bar{t},x)|^{2} \, dx & \leq  \left( \frac{1}{2^{N}} + \frac{1}{2^{N+1}}.... +
\frac{1}{2^{K-k}} \right) E R_{l_k}^{2} \\
& \leq \frac{1}{2^{N-1}} E R_{l_k}^{2}
\end{array}
\end{equation}
Now we let  $N = C^{'} \log{(g(M))} $ with $C^{'} \gg_{E} 1 $  so that $ \frac{E R_{l_k}^{2}}{2^{N-1}} \leq \frac{1}{8} c^{2}
g^{-\frac{34}{3}} (M) |J_{l_k}| $. By (\ref{Eqn:LowerBondMass2}) we have

\begin{equation}
\begin{array}{ll}
\sum\limits_{k^{'}=k+N}^{K} \int_{B(x_{_{l_{k^{'}}}} ,R_{l_k'} )} |u(\bar{t},x)|^{2} \, dx & \leq
\frac{1}{2}  \int_{B(x_{l_{k}},R_{l_k})} |u(\bar{t},x)|^{2} \, dx
\end{array}
\end{equation}
Therefore

\begin{equation}
\begin{array}{ll}
\int_{_{B(x_{l_{k}}, R_{l_k}) / \bigcup_{k^{'}=k+N}^{K} B (x_{l_{k^{'}}} ,R_{l_{k'}})  } } \   |u(\bar{t},x)|^{2} \, dx & \geq \frac{1}{2}  \int_{B(x_{l_{k}},R_{l_k})} |u(\bar{t},x)|^{2} \, dx \\
& \geq \frac{c^{2} g^{-\frac{34}{3}} (M)}{4} |J_{l_k}|
\end{array}
\end{equation}
and by H\"older inequality, there exists a positive constant $\ll_{E} 1$ (that we still denote by $c$) such that

\begin{equation}
\begin{array}{ll}
\int_{_{B(x_{l_{k}}, R_{l_k}) / \bigcup_{k^{'}=k+N}^{K} B (x_{l_{k^{'}}} ,R_{l_{k'}})  } } \   |u(\bar{t},x)|^{\frac{2n}{n-2}} \, dx & \geq c
g^{- \frac{8024}{3}} (M)
\end{array}
\end{equation}
and after summation over $k$, we

\begin{equation}
\begin{array}{l}
\frac{K}{N} c g^{-\frac{8024}{3}} (M) \lesssim E
\end{array}
\end{equation}
since $\sum\limits_{k=1}^{K} \chi_{ _{ B(x_{l_{k}}, R_{l_k}) / \cup_{k^{'}=k+N}^{K} B (x_{_{l_{k^{'}}}} ,R_{l_{k'}})}} \leq N$ and $\| u(t)
\|^{\frac{2n}{n-2}}_{L^{\frac{2n}{n-2}}} \lesssim  E$. Rearranging we see from (\ref{Eqn:LowerBoundK}) that there exist two constants $C_{1} \gg_{E} 1$ and $C_{2} \gg_{E} 1$ such that

\begin{equation}
\begin{array}{ll}
L & \leq \left( C_{1} g^{\frac{1972}{3}}(M) \right)^{ C_{2} \log{(g(M))}
g^{\frac{8024}{3}} (M) }
\end{array}
\label{Eqn:BoundL2}
\end{equation}
We see that (\ref{Eqn:BoundLFin}) holds. \\

\fbox{\textbf{Step 9}} \\

This is the final step. Recall that there are $L$ intervals $J_{l}$ and that on each of these intervals we have $\| u
\|_{L_{t}^{\frac{2(n+2)}{n-2}} L_{x}^{\frac{2(n+2)}{n-2}} (J) }^{\frac{2(n+2)}{n-2}} = \eta_{1}$. Therefore, there are two constants $\gg_{E} 1$
(that we denote by $C_{1}$ and $C_{2}$) such that (\ref{Eqn:BoundLong}) holds.

\end{proof}

\subsection{Proof of Lemma \ref{lem:regv}}
\label{Subsec:Lemmaregv}

In this subsection we prove Lemma \ref{lem:regv}. There are two cases

\begin{itemize}

\item $n=3$

By the fundamental theorem of calculus (and the inequality $ \| D u \|_{L_{t}^{\infty} L_{x}^{2} ([ t_{*} |J_{l}|, |J_{l}| ] )} \lesssim
E^{\frac{1}{2}}$ ) we have

\begin{equation}
\begin{array}{ll}
\| u_{h} - u \|_{L_{t}^{\infty} L_{x}^{2} ([ t_{*} |J_{l}|, |J_{l}| ] )} &  \leq E^{\frac{1}{2}} |h|
\end{array}
\label{Eqn:Diff1}
\end{equation}
Moreover, by Sobolev (and the inequality $ \| u \|_{L_{t}^{\infty} L_{x}^{6} ([ t_{*} |J_{l}|, |J_{l}| ] ) } \lesssim E^{\frac{1}{6}} $ ) we
have

\begin{equation}
\begin{array}{ll}
\| u_{h} -u \|_{ L_{t}^{\infty} L_{x}^{6} ([ t_{*} |J_{l}|, |J_{l}| ] )} \leq E^{\frac{1}{6}}
\end{array}
\label{Eqn:Diff2}
\end{equation}
Therefore, by interpolation of (\ref{Eqn:Diff1}) and (\ref{Eqn:Diff2}), we get

\begin{equation}
\begin{array}{ll}
\| u_{h} - u \|_{ L_{t}^{\infty} L_{x}^{3} ([ t_{*} |J_{l}|, |J_{l}| ] )} & \leq E^{\frac{1}{3}} |h|^{\frac{1}{2}}
\end{array}
\end{equation}
Now, by the fundamental theorem of calculus, the inequality $|x| g^{'}(|x|) \lesssim g(|x|)$, (\ref{Eqn:EquivF}) and (\ref{Eqn:EnergyBarely}) we
have

\begin{equation}
\begin{array}{ll}
\|  |u(s)|^{\frac{4}{n-2}} u(s) g(|u(s)|) - |u_{h}(s)|^{\frac{4}{n-2}} u_{h}(s) g(|u_{h}(s)|) \|_{L^{1}} & \lesssim \|
u_{h}(s)- u(s) \|_{L^{3}} \| u(s) g^{\frac{n-2}{2n}}(|u(s)|) \|^{4}_{L^{6}} \\
 &  \| g^{\frac{n-2}{n}}(|u(s)|) \|_{L^{\infty}} \\
& \lesssim_{E} g^{\frac{n-2}{n}}(M) |h|^{\frac{1}{2}}
\end{array}
\end{equation}
and, by the dispersive inequality (\ref{Eqn:DispIneq})  we conclude that

\begin{equation}
\begin{array}{ll}
\| v_{1,h}  - v_{1} \|_{L_{t}^{\infty} L_{x}^{\infty} ([t_{*}|J_{l}, |J_{l}| ])} & \lesssim_{E} \eta_{3}^{-\frac{1}{2}} |J_{l}|^{-\frac{1}{2}}
g^{\frac{n-2}{n}}(M) |h|^{\frac{1}{2}}
\end{array}
\end{equation}
Interpolating this inequality with

\begin{equation}
\begin{array}{l}
\| v_{1,h} - v_{1} \|_{L_{t}^{\infty} L_{x}^{6} ([t_{*} |J_{l}|, |J_{l}|])} = \|  u_{l, (t_{*} - \eta_{3}) |J_{l}|,h} - u_{l,t_{1},h} -
(u_{l,(t_{*} - \eta_{3})|J_{l}|} - u_{l,t_{1}}  ) \|_{L_{t}^{\infty} L_{x}^{6}  ([t_{*} |J_{l}|, |J_{l}|])} \\
\lesssim E^{\frac{1}{2}}
\end{array}
\end{equation}
we get (\ref{Eqn:Regv}).

\item $n=4$ By the fundamental theorem of calculus we have

\begin{equation}
\begin{array}{ll}
\| v_{1,h} - v_{1} \|_{L_{t}^{\infty} L_{x}^{\frac{2(n+2)}{n-2}} ([t_{*} |J_{l}|, |J_{l}|]) } & \lesssim \| D v_{1} \|_{L_{t}^{\infty}
L_{x}^{\frac{2(n+2)}{n-2}} [t_{*} |J_{l}|, |J_{l}|]  } |h|
\end{array}
\end{equation}
But, by interpolation

\begin{equation}
\begin{array}{ll}
\|  D v_{1} \|_{L_{t}^{\infty} L_{x}^{\frac{2(n+2)}{n-2}} ( [t_{*} |J_{l}|, |J_{l}|] )  }  & \lesssim \| D v_{1} \|^{\frac{2}{n+2}}_{
L_{t}^{\infty} L_{x}^{2} ([t_{*} |J_{l}|, |J_{l}| ])}
\|  D v_{1} \|^{\frac{n}{n+2}}_{ L_{t}^{\infty} L_{x}^{\frac{2n}{n-4}}  ( [t_{*} |J_{l}|, |J_{l}|] ) } \\
& \lesssim_{E} \| D v_{1} \|^{\frac{n}{n+2}}_{ L_{t}^{\infty} L_{x}^{\frac{2n}{n-4}}  ([t_{*} |J_{l}|, |J_{l}|]) }
\end{array}
\label{Eqn:Dv1Interp}
\end{equation}
So it suffices to estimate $\| D v_{1} \|_{ L_{t}^{\infty} L_{x}^{\frac{2n}{n-4}}  ([t_{*} |I_{j}|, |I_{j}|]) }$. By (\ref{Eqn:EnergyBarely}),
(\ref{Eqn:EquivF}) and Result \ref{res:ControlDu} we have

\begin{equation}
\begin{array}{ll}
\| D  (|u|^{\frac{4}{n-2}} u g(|u|) ) \|_{ L_{s}^{\infty} L_{x}^{\frac{2n}{n+4}}  ( [t_{1}, (t_{*} - \eta_{3}) |J_{l}| ] ) } & \lesssim \| D u
\|_{ L_{s}^{\infty} L_{x}^{2} ([t_{1}, (t_{*} - \eta_{3}) |J_{l}| ] )  } \| u g^{\frac{n-2}{2n}} (|u|)
\|^{\frac{4}{n-2}}_{L_{s}^{\infty} L_{x}^{\frac{2n}{n-2}} [t_{1}, (t_{*} - \eta_{3}) |J_{l}| ]} \\
& g^{\frac{n-2} {n}} ( \| u \|_{L_{t}^{\infty} \tilde{H}^{k} ( [t_{1}, (t_{*} - \eta_{3}) |J_{l}| ] ) }) \\
& \lesssim_{E}  g^{\frac{n-2}{n}}(M)
\end{array}
\label{Eqn:EstInterm}
\end{equation}
and by combining (\ref{Eqn:EstInterm}) with the dispersive inequality (\ref{Eqn:DispIneq}) we have

\begin{equation}
\begin{array}{ll}
\| D v_{1} \|_{L_{t}^{\infty} L_{x}^{\frac{2n}{n-4}}  ([t_{*} |J_{l}|, |J_{l}|]) } & \lesssim \left\| \int_{t_{1}}^{(t_{*} - \eta_{3})|J_{l}|}
\| D e^{i(t-s) \triangle} ( |u(s)|^{\frac{4}{n-2}} u(s)  g(|u(s)|) ) \|_{L_{x}^{\frac{2n}{n-4}}} \, ds \right\|_{L_{t}^{\infty} ([t_{*} |J_{l}|, |J_{l}|])  } \\
& \lesssim \left\| \int_{t_{1}}^{(t_{*} - \eta_{3})|J_{l}|}  \frac{1}{|t - s|^{2}} \| D ( |u(s)|^{\frac{4}{n-2}} u(s)  g(|u(s)|))_{
L_{x}^{\frac{2n}{n+4}} } \, ds \right\|_{L_{t}^{\infty} ( [t_{*}|J_{l}|, |J_{l}|]  ) } \\
& \lesssim g^{\frac{n-2}{n}} (M) \eta_{3}^{-1} |J_{l}|^{-1}
\end{array}
\label{Eqn:ControlDv1}
\end{equation}
We conclude from (\ref{Eqn:Dv1Interp}) and (\ref{Eqn:ControlDv1}) that (\ref{Eqn:Regv}) holds.

\end{itemize}

\subsection{Proof of Lemma \ref{lem:Morawest}}

By (\ref{Eqn:BarelySchrod}) we have \footnote{Throughout this subsection, all the computations are done for smooth solutions. Then (\ref{Eqn:MorawEst}) holds for an $\tilde{H}^{k}-$ solution by a standard approximation argument with smooth solutions.}

\begin{equation}
\begin{array}{ll}
\partial_{t} \Im (\partial_{k} u \bar{u}) & = \Re \left[ |u|^{\frac{4}{n-2}} \bar{u} g(|u|) \partial_{k} u  -
\partial_{k} ( |u|^{\frac{4}{n-2}} u g(|u|)) \right] + \Re \left( \triangle (\partial_{k} u) \bar{u} -
\overline{\triangle u} \partial_{k} u \right)
\end{array}
\label{Eqn:1Mor}
\end{equation}
Moreover

\begin{equation}
\begin{array}{ll}
\frac{1}{2} \partial_{k} \triangle (|u|^{2}) & = 2 \partial_{j} \Re (\partial_{k} u \overline{\partial_{j} u}) - \Re (\partial_{k} u \triangle
\bar{u}) + \Re (u \triangle \overline{\partial_{k} u})
\end{array}
\label{Eqn:2Mor}
\end{equation}
Therefore, adding (\ref{Eqn:1Mor}) and (\ref{Eqn:2Mor}) leads to

\begin{equation}
\begin{array}{ll}
\partial_{t} \Im (\partial_{k} u \bar{u}) & = -2 \partial_{j} \Re (\partial_{k} u \overline{\partial_{j} u})
+ \frac{1}{2} \partial_{k} \triangle (|u|^{2}) + \Re \left[ |u|^{\frac{4}{n-2}} \bar{u} g(|u|) \partial_{k} u  -
\partial_{k} ( |u|^{\frac{4}{n-2}} u g(|u|)) \bar{u} \right]
\end{array}
\end{equation}
It remains to understand $ \Re \left[ |u|^{\frac{4}{n-2}} \bar{u} g(|u|) \partial_{k} u  -
\partial_{k} ( |u|^{\frac{4}{n-2}} u g(|u|)) \bar{u} \right]$. We write

\begin{equation}
\begin{array}{ll}
\Re \left[ |u|^{\frac{4}{n-2}} \bar{u} g(|u|) \partial_{k} u  -
\partial_{k} ( |u|^{\frac{4}{n-2}} u g(|u|)) \bar{u} \right] & = A_{1} + A_{2}
\end{array}
\end{equation}
with

\begin{equation}
\begin{array}{ll}
A _{1}: = & \Re \left[ |u|^{\frac{4}{n-2}} \bar{u} g(|u|) \partial_{k} u \right]
\end{array}
\end{equation}
and

\begin{equation}
\begin{array}{ll}
A_{2}:= & - \Re \left( \partial_{k} (|u|^{\frac{4}{n-2}} u g(|u|) ) \bar{u} \right)
\end{array}
\end{equation}
We are interested in finding a function $F_{1}: \mathbb{C} \times \mathbb{C} \rightarrow \mathbb{C} $, continuouly differentiable  such that
$F_{1}(z,\bar{z})= \overline{F_{1}(z,\bar{z})}$, $F_{1}(0,0)=0$ and $A_{1} = \partial_{k} F_{1}(u,\bar{u})$. Notice that the first condition
implies in particular that $\partial_{\bar{z}} F_{1}(z,\bar{z})= \overline{\partial_{z} F_{1}(z,\bar{z})}$. Therefore we get, after computation

\begin{equation}
\begin{array}{ll}
\partial_{z} F_{1}(z, \bar{z}) & = \frac{ |z|^{\frac{4}{n-2}} \bar{z} g(|z|)}{2} \\
\partial_{\bar{z}} F_{1}(z,\bar{z}) & = \frac{ |z|^{\frac{4}{n-2}} z g(|z|) }{2}
\end{array}
\end{equation}
and by the fundamental theorem of calculus, if such a function exists, then

\begin{equation}
\begin{array}{ll}
F_{1}(z,\bar{z}) & = \int_{0}^{1} F_{1}^{'}(tz, t \bar{z}) \cdot (z,\bar{z}) \, dt \\
& =  2 \Re \int_{0}^{1} \partial_{z} F_{1} (tz, t \bar{z}) z \, dt \\
& = \int_{0}^{1} |t z|^{\frac{4}{n-2}} t |z|^{2} g(t |z|) \, dt \\
\end{array}
\end{equation}
and, after a change of variable, we get

\begin{equation}
\begin{array}{ll}
F_{1}(z,\bar{z}) & = \int_{0}^{|z|} t^{\frac{n+2}{n-2}} g(t) \, dt
\end{array}
\end{equation}
Conversely it is not difficult to see that $F_{1}$ satisfies all the required conditions.

We turn now to $A_{2}$. We can write

\begin{equation}
\begin{array}{ll}
A_{2} & =A_{2,1} + A_{2,2}
\end{array}
\end{equation}
with

\begin{equation}
\begin{array}{ll}
A_{2,1} & := - \Re \left( \partial_{u} ( |u|^{\frac{4}{n-2}} u g(|u|) ) \bar{u} \partial_{k} u    \right)
\end{array}
\end{equation}
and

\begin{equation}
\begin{array}{ll}
A_{2,2} & : = - \Re \left( \partial_{\bar{u}} ( |u|^{\frac{4}{n-2}} u g(|u|) ) \bar{u} \partial_{k} u     \right)
\end{array}
\end{equation}
Again we search for a function $F_{2,1}: \mathbb{C} \times  \mathbb{C} \rightarrow \mathbb{C}$ and continuously differentiable such that
$F_{2,1}(z,\bar{z})= \overline{F_{2,1}(z,\bar{z})}$ and $A_{2,1}= \partial_{k} F_{2,1} (u,\bar{u})$. By identification we have

\begin{equation}
\begin{array}{ll}
\partial_{z} F_{2,1}(z,\bar{z}) & = - \frac{  |z|^{\frac{4}{n-2}} \bar{z} \left( \left( \frac{2}{n-2} + 1 \right) g(|z|) + \frac{g^{'}(|z|) |z|}{2}    \right)  }
{2} \\
\partial_{\bar{z}} F_{2,1} (z,\bar{z}) & = - \frac{ |z|^{\frac{4}{n-2}} z  \left( \left( \frac{2}{n-2} + 1 \right)  g(|z|) + \frac{g^{'}(|z|) |z|}{2} \right) } {2}
\end{array}
\end{equation}
and by the fundamental theorem of calculus

\begin{equation}
\begin{array}{ll}
F_{2,1}(z,\bar{z})  & = \int_{0}^{1} F_{2,1}^{'}(tz, t \bar{z}) \cdot (z,\bar{z}) \,dt \\
& = \int_{0}^{1} 2 \Re \left( \partial_{z} F_{2,1} (tz, t \bar{z}) z  \right) \, dt \\
& = - \int_{0}^{1} |t z|^{\frac{4}{n-2}} \left(  \left(\frac{2}{n-2}+ 1 \right) g(|tz|) + \frac{g^{'}(|tz|) |tz|}{2} \right) t |z|^{2} \, dt \\
\end{array}
\end{equation}
and, after a change of variable, we get

\begin{equation}
\begin{array}{ll}
F_{2,1} (z,\bar{z}) & = - \int_{0}^{|z|} t^{\frac{n+2}{n-2}} \left( \left( \frac{2}{n-2} + 1 \right) g(t) + \frac{t g^{'} (t)}{2} \right) \, dt
\end{array}
\end{equation}
Again, we can easily check that $F_{2,1}$ satisfies all the required conditions. By using a similar process we can prove that

\begin{equation}
\begin{array}{ll}
A_{2,2} & = \partial_{k} F_{2,2} (u,\bar{u})
\end{array}
\end{equation}
with

\begin{equation}
\begin{array}{ll}
F_{2,2}(z,\bar{z}) & = - \int_{0}^{|z|} t^{\frac{n+2}{n-2}} \left( \frac{2}{n-2} g(t) + \frac{t g^{'}(t)}{2} \right) \, dt
\end{array}
\end{equation}
Therefore we get the local momentum conservation identity

\begin{equation}
\begin{array}{ll}
\partial_{t} \Im (\partial_{k} u \bar{u}) & = -2 \partial_{j} \Re (\partial_{k} u \overline{\partial_{j} u})
+ \frac{1}{2} \partial_{k} \triangle (|u|^{2}) - \partial_{k} \left(  \tilde{F}(u,\bar{u}) \right)
\end{array}
\label{Eqn:LocalMomentumId}
\end{equation}
with $ \tilde{F}(u,\bar{u}) $ defined in (\ref{Eqn:DeftildeF}). This identity has a similar structure to the local momentum conservation that
for a solution $v$ of the energy-critical Schr\"odinger equation

\begin{equation}
\begin{array}{ll}
\partial_{t} \Im \left( \partial_{k} v \bar{v} \right) & = -2 \partial_{j} \Re (\partial_{k} v \overline{\partial_{j} v} ) + \frac{1}{2}
\partial_{k} \triangle (|v|^{2}) + \partial_{k} \left( -\frac{2}{n} |u|^{\frac{2n}{n-2}}    \right)
\end{array}
\end{equation}
With this in mind, we multiply (\ref{Eqn:LocalMomentumId}) by an appropriate spatial cutoff, in the same spirit as Bourgain \cite{bour} and
Grillakis \cite{grill}, to prove a  Morawetz-type estimate. We follow closely an argument of Tao \cite{taorad}: we introduce the weight $a(x):=
\left( \epsilon^{2} +   \left( \frac{|x|}{A |I|^{\frac{1}{2}}} \right)^{2} \right)^{\frac{1}{2}} \chi \left( \frac{x}{A |I|^{\frac{1}{2}}}
\right) $ where $\chi$ is a smooth function,radial such that $\chi(|x|)=1$ for $|x| \leq 1$ and $\chi(|x|)=0$ for $|x| \geq 2$. We give here the
details since this equation, unlike the energy-critical Schr\"odinger equation, has no scaling property. Notice that $a$ is convex on $|x| \leq
A |I|^{\frac{1}{2}}$ since it is a composition of two convex functions. We multiply (\ref{Eqn:LocalMomentumId}) by $\partial_{k} a$ and we
integrate by parts

\begin{equation}
\begin{array}{ll}
\partial_{t} \int_{\mathbb{R}^{n}} \partial_{k} a \Im (\partial_{k} u \bar{u}) & =
 2 \int_{\mathbb{R}^{n}} \partial_{j} \partial_{k} a \Re (\partial_{k} u \overline{\partial_{j} u} )
 - \frac{1}{2} \int_{\mathbb{R}^{n}} \triangle (\triangle a)  |u|^{2} \, dx + \int_{\mathbb{R}^{n}} \triangle a \tilde{F} (u,\bar{u}) (t,x) \, dx
\end{array}
\end{equation}
A computation shows that for $ 0 \leq |x| \leq A |I|^{\frac{1}{2}}$

\begin{equation}
\begin{array}{ll}
\triangle a & = \frac{n-1}{( A |I|^{\frac{1}{2}} )^{2}} \left( \epsilon^{2} + \frac{|x|^{2}}{(A |I|^{\frac{1}{2}})^{2}}  \right)^{-\frac{1}{2}}
+ \frac{\epsilon^{2}}{(A |I|^{\frac{1}{2}})^{2}} \left( \epsilon^{2} + \frac{|x|^{2}}{(A |I|^{\frac{1}{2}})^{2}} \right)^{-\frac{3}{2}}
\end{array}
\end{equation}
and

\begin{equation}
\begin{array}{ll}
- \triangle \triangle a & = \frac{(n-1)(n-3)}{(A |I|^{\frac{1}{2}})^{4}} \left( \epsilon^{2} + \frac{|x|^{2}}{(A |I|^{\frac{1}{2}})^{2}}
\right)^{-\frac{3}{2}} + \frac{6(n-3) \epsilon^{2}}{(A |I|^{\frac{1}{2}})^{4}} \left( \epsilon^{2} + \frac{|x|^{2}}{(A |I|^{\frac{1}{2}})^{2}}
\right)^{-\frac{5}{2}} + \frac{15 \epsilon^{4}}{ (A |I|^{\frac{1}{2}})^{4}} \left( \epsilon^{2} + \frac{|x|^{2}}{(A |I|^{\frac{1}{2}})^{2}}
\right)^{-\frac{7}{2}}
\end{array}
\end{equation}
Moreover we have  $ \left| - \triangle(\triangle a) \right|  \lesssim \frac{1}{(A |I|^{\frac{1}{2}})^{4}}$, $ \left| \triangle a \right|
\lesssim \frac{1}{(A |I|^{\frac{1}{2}})^{2}} $ and $| \partial_{j} \partial_{k} a | \lesssim \frac{1}{(A |I|^{\frac{1}{2}})^{2}}$ for $ A |I|^{\frac{1}{2}}
\leq |x| \leq 2 A |I|^{\frac{1}{2}}$ and $|\partial_{k} a | \lesssim \frac{1}{A |I|^{\frac{1}{2}}} $ for $|x| \leq 2 A |I|^{\frac{1}{2}}$.
Therefore by the previous estimates, (\ref{Eqn:EnergyBarely}), (\ref{Eqn:EquivF}) and the inequality $|x| g^{'}(|x|) \lesssim g(|x|)$ we get,
after integrating on $I \times \mathbb{R}^{n}$ and letting $\epsilon$ go to zero

\begin{equation}
\begin{array}{ll}
 \frac{1}{A |I|^{\frac{1}{2}}}\int_{I} \int_{|x| \leq A |I|^{\frac{1}{2}} } \frac{ \tilde{F} (u,\bar{u}) (t,x)}{|x|} \, dx \, dt
- C   (A |I|^{\frac{1}{2}})^{-2} E |I| - C (A |I|^{\frac{1}{2}})^{-4} E (A |I|^{\frac{1}{2}})^{2} |I| \lesssim E
\end{array}
\end{equation}
for some constant $C \geq 1$. After rearranging we get (\ref{Eqn:MorawEst}).

\section{APPENDIX A}
We shall prove the following Leibnitz rule:

\begin{prop}{\textbf{``A fractional Leibnitz rule''}}
Let $ 0 \leq \alpha < 1$, $k$ and $\beta$ be integers such that $k \geq 2$ and $\beta > k-1 $,  $(r , r_{2}) \in (1,\infty)^{2}$,
$(r_{1},r_{3}) \in (1, \infty]^{2}$ be such that $\frac{1}{r}=
\frac{\beta}{r_{1}} + \frac{1}{r_{2}} +\frac{1}{r_{3}}$. Let $F: \mathbb{R}^{+} \rightarrow \mathbb{R}$ be a $C^{k}$- function and let $G:=\mathbb{R}^{2} \rightarrow \mathbb{R}^{2}$ be a $C^{k}$- function such that

\begin{equation}
\begin{array}{l}
F^{[i]}(x) =O \left( \frac{F(x)}{x^{i}} \right), \; \tau \in [0,1]: \;
\left| F \left( |\tau x + (1-\tau)y|^{2} \right) \right|
\lesssim \left| F(|x|^{2}) \right| + \left| F(|y|^{2}) \right|,
\end{array}
\label{Eqn:Cdtionf}
\end{equation}
and

\begin{equation}
G^{[i]}(x,\bar{x})  = O (|x|^{\beta + 1 -i})
\label{Eqn:CdtionG}
\end{equation}
for $ 0 \leq i \leq k$. Then

\begin{equation}
\begin{array}{ll}
\left\| D^{ k -1 + \alpha} ( G(f,\bar{f}) F(|f|^{2}) \right\|_{L^{r}} & \lesssim \| f \|^{\beta}_{L^{r_{1}}} \| D^{k -1  + \alpha} f \|_{L^{r_{2}}}
\| F(|f|^{2}) \|_{L^{r_{3}}}
\end{array}
\label{Eqn:EstToProveFrac}
\end{equation}
Here $F^{[i]}$ and $G^{[i]}$ denote the $i^{th}$- derivatives of $F$ and $G$ respectively. \\
More generally, let $\tilde{F}: \mathbb{R}^{+} \rightarrow \mathbb{R}$ be a $C^{k}$ function. Substitute $F$ with
$\tilde{F}$ on the right-hand side of the equality of (\ref{Eqn:Cdtionf}), in the inequality of (\ref{Eqn:Cdtionf}), and on the right-hand side of
(\ref{Eqn:EstToProveFrac}). With these substitutions made, if
$F$, $\tilde{F}$, and $G$ satisfy (\ref{Eqn:Cdtionf}) and (\ref{Eqn:CdtionG}), then $F$ and $G$ satisfy (\ref{Eqn:EstToProveFrac}).
 \label{Prop:FracLeibn}

\end{prop}

\begin{proof}
The proof relies upon an induction process, the usual product rule for fractional derivatives

\begin{equation}
\begin{array}{ll}
\| D^{\alpha_{1}} (fg) \|_{L^{q}} & \lesssim \| D^{\alpha_{1}} f \|_{L^{q_{1}}} \| g \|_{L^{q_{2}}} + \| f \|_{L^{q_{3}}} \| D^{\alpha_{1}} g
\|_{L^{q_{4}}}
\end{array}
\label{Eqn:FracProd}
\end{equation}
and the usual Leibnitz rule for fractional derivatives :

\begin{equation}
\begin{array}{ll}
\| D^{\alpha_{2}} H(f) \|_{L^{q}} & \lesssim \| \tilde{H}(f) \|_{L^{q_{1}}}  \| D^{\alpha_{2}} f \|_{L^{q_{2}}}
\end{array}
\label{Eqn:DerivComp}
\end{equation}
if $H$ is $C^{1}$ and it satisfies $ \tau \in [0,1]: \, \left| H^{'}\left( \tau x + (1- \tau) y \right) \right| \lesssim
\tilde{H}(x) + \tilde{H}(y)$, $ 0 \leq \alpha_{1} < \infty$, $ 0 < \alpha_{2}  \leq 1 $, $ (q,q_4) \in (1, \infty)^{2}$, $q_{3} \in (1,\infty]$,
$(q_1,q_2) \in (1,\infty) \times (1,\infty] $ in (\ref{Eqn:FracProd}), $(q_1,q_2) \in (1,\infty] \times (1,\infty) $ in (\ref{Eqn:DerivComp}),
$\frac{1}{q}= \frac{1}{q_{1}} + \frac{1}{q_{2}}$, and
$\frac{1}{q}= \frac{1}{q_3} + \frac{1}{q_4}$ (see e.g. Christ-Weinstein \cite{christwein}, Taylor
\cite{taylor} and references in \cite{taylor})
\footnote{Notice that in \cite{christwein}, they add the restriction $ 0 < \alpha_{1} < 1$. It is not difficult to see that this restriction is
not necessary: see Taylor \cite{taylor} for example}. Moreover we shall use interpolation and the properties of $F$ to control the intermediate
terms.

Let $k=2$. Then

\begin{equation}
\begin{array}{ll}
\left\| D^{2-1 +\alpha} ( G(f,\bar{f}) F(|f|^{2}) ) \right\|_{L^{r}} &  \sim \left\| D^{\alpha } \nabla ( G(f,\bar{f}) F(|f|^{2})  ) \right\|_{L^{r}} \\
& \lesssim \left\| D^{\alpha}( \partial_{z} G(f,\bar{f}) \nabla f  F(|f|^{2}) )  \right\|_{L^{r}} + \left\| D^{\alpha} ( \partial_{\bar{z}}
G(f,\bar{f}) \overline{ \nabla f}  F(|f|^{2})  ) \right\|_{L^{r}} \\
& + \left\| D^{\alpha} \left( F^{'}(|f|^{2}) \left( 2 \Re \left( \bar{f}  \nabla f \right)  G(f,\bar{f}) \right) \right) \right\|_{L^{r}}  \\
& \lesssim  A_{1} + A_{2} + A_{3}
\end{array}
\end{equation}
We estimate $A_{1}$. $A_{2}$ is estimated in a similar fashion. By (\ref{Eqn:FracProd}), (\ref{Eqn:DerivComp}) and the assumption
(\ref{Eqn:Cdtionf})

\begin{equation}
\begin{array}{l}
A_{1} \lesssim \| D^{\alpha} ( \partial_{z} G(f,\bar{f})  F(|f|^{2}) )  \|_{L^{r_{4}}} \| D f \|_{L^{r_{5}}} +
\| \partial_{z} G(f,\bar{f}) F(|f|^{2}) \|_{L^{r_{6}}} \| D^{ (2-1) + \alpha} f \|_{L^{r_{2}}} \\
\lesssim \| f \|^{\beta -1}_{L^{r_{1}}} \| F(|f|^{2}) \|_{L^{r_{3}}} \| D^{\alpha} f \|_{L^{r_{8}}} \| D f  \|_{L^{r_{5}}} + \| f
\|_{L^{r_{1}}}^{\beta} \| D^{(2-1)+\alpha} f \|_{L^{r_{2}}} \| F(|f|^{2}) \|_{L^{r_{3}}}
\end{array}
\label{Eqn:EstFrac1}
\end{equation}
with $\frac{1}{r}= \frac{1}{r_{4}} + \frac{1}{r_{5}}$, $\frac{1}{r}= \frac{1}{r_{6}} + \frac{1}{r_{2}} $, $\frac{1}{r_{4}} = \frac{\beta
-1}{r_{1}} + \frac{1}{r_{3}} + \frac{1}{r_{8}}$, $\frac{1}{r_{5}}= \frac{1 - \theta_{1}}{r_{1}} + \frac{\theta_{1}}{r_{2}}$ and $\theta_{1} =
\frac{1}{1+ \alpha}$. Notice that these relations imply that $\frac{1}{r_{8}}= \frac{\theta_{1}}{r_{1}} + \frac{1 - \theta_{1}}{r_{2}}$. Now, by
complex interpolation, we have

\begin{equation}
\begin{array}{ll}
\| D^{\alpha} f \|_{L^{r_{8}}} & \lesssim \| f \|^{\theta_{1}}_{L^{r_{1}}} \| D^{ (2-1) + \alpha}  f \|^{1- \theta_{1}}_{L^{r_{2}}}
\end{array}
\label{Eqn:Interp11}
\end{equation}
and

\begin{equation}
\begin{array}{ll}
 \| D f  \|_{L^{r_{5}}} & \lesssim \| f \|^{1- \theta_{1}}_{L^{r_{1}}} \| D^{ (2-1) + \alpha} f  \|^{\theta_{1}}_{L^{r_{2}}}
\end{array}
\label{Eqn:Interp12}
\end{equation}
Plugging (\ref{Eqn:Interp11}) and (\ref{Eqn:Interp12}) into (\ref{Eqn:EstFrac1}) we get (\ref{Eqn:EstToProveFrac}).

We estimate $A_{3}$.

\begin{equation}
\begin{array}{ll}
A_{3} & \lesssim \sum \limits_{\tilde{f} \in \{ f,\bar{f} \}}  \left\| D^{\alpha} \left( F^{'}(|f|^{2})  \tilde{f} G(f,\bar{f}) \right) \right\|_{L^{r_{4}}} \| D f \|_{L^{r_{5}}} + \|
D^{\alpha +1} f \|_{L^{r_{2}}} \left\| F^{'}(|f|^{2}) \tilde{f}  G(f,\bar{f}) \right\|_{L^{r_{6}}} \\
& \lesssim A_{3,1} + A_{3,2}
\end{array}
\end{equation}
Using the assumption $F^{'}(|x|^{2}) = O \left( \frac{F(|x|^{2})}{|x|^{2}}   \right)$ we get $ A_{3,2} \lesssim \| f \|^{\beta}_{L^{r_{1}}}
\| D^{1+ \alpha} f \|_{L^{r_{2}}} \| F(|f|^{2}) \|_{L^{r_{3}}} $. Moreover, by (\ref{Eqn:DerivComp}), the assumptions on $F$ and $G$, (\ref{Eqn:Interp11}) and
(\ref{Eqn:Interp12}) we get

\begin{equation}
\begin{array}{ll}
A_{3,1} & \lesssim \| F(|f|^{2}) |f|^{\beta -1}  \|_{L^{r_{7}}} \| D^{\alpha} f \|_{L^{r_{8}}} \| D f \|_{L^{r_{5}}} \\
& \lesssim  \| f \|^{\beta}_{L^{r_{1}}}  \| D^{1+ \alpha} f \|_{L^{r_{2}}} \| F(|f|^{2}) \|_{L^{r_{3}}}
\end{array}
\end{equation}
with $\frac{1}{r_{7}} + \frac{1}{r_{8}}= \frac{1}{r_{4}}$. The more general statement follows exactly the same steps and its proof
is left to the reader. \\
Now let us assume that the result is true for $k$. Let us prove that it is also true for $k+1$. By (\ref{Eqn:FracProd}) we have

\begin{equation}
\begin{array}{ll}
\| D^{k + \alpha} (G(f,\bar{f}) F(|f|^{2})) \|_{L^{r}} & \sim \| D^{k-1 + \alpha} \nabla ( G(f,\bar{f}) F(|f|^{2}) )  \|_{L^{r}} \\
& \lesssim \| D^{k-1 + \alpha} \partial_{z} G(f,\bar{f}) \nabla f F(|f|^{2}) \|_{L^{r}} + \| D^{k-1 + \alpha} \partial_{\bar{z}} G(f,\bar{f})
\overline{\nabla f} F(|f|^{2}) \|_{L^{r}} \\
& + \left\| D^{k-1 + \alpha} \left[ G(f,\bar{f}) F^{'}(|f|^{2})  \left( 2 \Re \left( \bar{f} \nabla f \right) \right) \right]  \right\|_{L^{r}} \\
& \lesssim A^{'}_{1} + A^{'}_{2} + A^{'}_{3}
\end{array}
\end{equation}
We estimate $A^{'}_{1}$ and $A^{'}_{3}$. $A^{'}_{2}$ is estimated in a similar fashion as $A^{'}_{1}$. By (\ref{Eqn:FracProd}),
(\ref{Eqn:DerivComp}) and the assumption  $| \partial_{z} G (f,\bar{f}) | \lesssim |f|^{\beta}$ we have

\begin{equation}
\begin{array}{ll}
A^{'}_{1} & \lesssim \| D^{k + \alpha} f  \|_{L^{r_{2}}}  \| \partial_{z} G(f,\bar{f}) F(|f|^{2}) \|_{L^{r_{6}}} + \| D^{k - 1 + \alpha} (
\partial_{z} G(f,\bar{f}) F(|f|^{2}) ) \|_{L^{r^{'}_{4}}}  \| D f \|_{L^{r^{'}_{5}}} \\
& \lesssim \| f \|^{\beta}_{L^{r_{1}}} \| D^{(k + 1) - 1 + \alpha} f \|_{L^{r_{2}}}  \| F(|f|^{2}) \|_{L^{r_{3}}} + A^{'}_{1,1}
\end{array}
\label{Eqn:EstAprime1}
\end{equation}
with $r^{'}_{4}$, $r^{'}_{5}$ such that $ \frac{1}{r^{'}_{4}} + \frac{1}{r^{'}_{5}} = \frac{1}{r} $, $\frac{1}{r^{'}_{5}} = \frac{1-
\theta_{1}^{'}}{r_{1}} + \frac{\theta_{1}^{'}}{r_{2}}$ and $\theta^{'}_{1}= \frac{1}{k + \alpha}$. Notice that, since we assumed that the result
is true for $k$, we get, after checking that  $\partial_{z} G$ satisfies the right assumptions

\begin{equation}
\begin{array}{ll}
\| D^{k - 1 + \alpha} ( \partial_{z} G(f,\bar{f}) F(|f|^{2}) )  \|_{L^{r^{'}_{4}}} & \lesssim \| f \|_{L^{r_{1}}}^{\beta -1} \| D^{k - 1 + \alpha} f
\|_{L^{r^{'}_{8}}}\| F(|f|^{2}) \|_{L^{r_{3}}}
\end{array}
\label{Eqn:InductionIneq}
\end{equation}
with $r^{'}_{8}$ such that $\frac{1}{r^{'}_{4}} = \frac{\beta -1}{r_{1}} + \frac{1}{r^{'}_{8}} + \frac{1}{r_{3}}$. Notice also that, by complex
interpolation

\begin{equation}
\begin{array}{ll}
\| D f \|_{L^{r^{'}_{5}}} & \lesssim \| f \|^{ 1 - \theta^{'}_{1}}_{L^{r_{1}}} \| D^{(k+1)- 1 + \alpha} f \|^{\theta^{'}_{1}}_{L^{r_{2}}}
\end{array}
\label{Eqn:Interp1prime1}
\end{equation}
and

\begin{equation}
\begin{array}{ll}
\| D^{k-1 + \alpha} f \|_{L^{r^{'}_{8}}} & \lesssim \| f \|^{\theta^{'}_{1}}_{L^{r_{1}}} \| D^{(k + 1) - 1 + \alpha} f \|^{ 1 -
\theta^{'}_{1}}_{L^{r_{2}}}
\end{array}
\label{Eqn:Interp1prime2}
\end{equation}
Combining (\ref{Eqn:InductionIneq}), (\ref{Eqn:Interp1prime1}) and (\ref{Eqn:Interp1prime2}) we have

\begin{equation}
\begin{array}{ll}
A^{'}_{1,1} & \lesssim \| f \|^{\beta}_{L^{r_{1}}} \| D^{k + \alpha} f \|_{L^{r_{2}}}  \| F(|f|^{2}) \|_{L^{r_{3}}}
\end{array}
\end{equation}
Plugging this bound into (\ref{Eqn:EstAprime1}) we get the required bound for $A_{1,1}^{'}$.

We turn to $A^{'}_{3}$. Let $\tilde{F}(x):=x F^{'}(x)$. From the induction assumption applied to $\tilde{F}$ we get

\begin{equation}
\begin{array}{ll}
A^{'}_{3} & \lesssim \sum \limits_{\tilde{f} \in \{f,\bar{f}\}}  \left\| D^{k-1 + \alpha} \left[ G(f,\bar{f}) F^{'}(|f|^{2}) \tilde{f} \right]  \right\|_{L^{r^{'}_{4}}} \| D f
\|_{L^{r^{'}_{5}}} + \| D^{k + \alpha} f
\|_{L^{r_{2}}} \| G(f,\bar{f}) F^{'}(|f|^{2}) \|_{L^{r_{6}}} \\
& \lesssim \| f \|^{\beta - 1}_{L^{r_{1}}} \| D^{k-1 + \alpha} f \|_{L^{r_{8}^{'}}}
\|F(|f|^{2}) \|_{L^{r_{3}}} \| D f \|_{L^{r_{5}^{'}}} + \| D^{k + \alpha} f
\|_{L^{r_{2}}} \| f
\|^{\beta}_{L^{r_{1}}} \| F(|f|^{2}) \|_{L^{r_{3}}} \\
& \lesssim \| f \|^{\beta}_{L^{r_{1}}} \| D^{k + \alpha} f \|_{L^{r_{2}}} \| F(|f|^{2}) \|_{L^{r_{3}}}
\end{array}
\end{equation}
Again the more general statement follows exactly the same steps and its proof
is left to the reader.

\end{proof}

\section{APPENDIX B}
We shall prove the following proposition:

\begin{prop}
Let $\lambda \in \mathbb{N}^{*}$  and $(Q,R)$ be such that
$\left( \frac{1}{Q}, \frac{1}{R} \right) =  \left( \frac{(\lambda -1) (n-2)}{2(n+2)} + \frac{n}{2(n+2)} \right) (1,1)$.
Let $J$ be an interval. Let $k > \frac{n}{2}$. Let
$ \bar{Q}_{k}(J,u)  := \| u \|_{L_{t}^{\infty} \tilde{H}^{k}(J)} +
\| D u \|_{L_{t}^{\frac{2(n+2)}{n}} L_{x}^{\frac{2(n+2)}{n}}(J)}
+ \| D^{k} u \|_{L_{t}^{\frac{2(n+2)}{n}} L_{x}^{\frac{2(n+2)}{n}}(J)} $.  Let $\psi(k)$ be defined
as follows: if $ \frac{n}{2} < k < \frac{n+2}{n-2}$ then $\psi(k) := k$ and if $k \geq \frac{n+2}{n-2}$ then
$\psi(k) := \frac{n+2}{n-2}-$. There exists $\bar{C} >0$ such that

\begin{equation}
\begin{array}{ll}
\left\| D^{k}(u^{\lambda} g(|u|) ) \right\|_{L_{t}^{Q} L_{x}^{R}(J)} & \lesssim
\| u \|^{\lambda-1}_{L_{t}^{\frac{2(n+2)}{n-2}} L_{x}^{\frac{2(n+2)}{n-2}}(J)} \bar{Q}_{k}(J,u)
\langle \bar{Q}_{\frac{n}{2}+} (J,u) \rangle^{\bar{C}} \\
&
\left(
\begin{array}{l}
g \left( \bar{Q}_{k}(J,u) \right) + \langle \bar{Q}_{\psi(k)}(J,u) \rangle^{\bar{C}}
+ \langle \bar{Q}_{k- \frac{1}{4}} (J,u) \rangle^{\bar{C}}
\end{array}
\right) \cdot
\end{array}
\label{Eqn:EstHighReg}
\end{equation}
The same estimate holds if $u^{\lambda}$ is replaced with $u^{\lambda_1} \bar{u}^{\lambda_2}$ with $(\lambda_1,\lambda_2) \in \mathbb{N}^{2}$ such that $\lambda_1 + \lambda_2 = \lambda$, or if $g(|u|)$ is replaced with
$\tilde{g}^{'}(|u|^{2}) u^{\lambda_3} \bar{u}^{\lambda_4}$ with $(\lambda_3,\lambda_4) \in \mathbb{N}^{2}$ such that
$\lambda_3 + \lambda_4 = 2$ and  $\tilde{g}(x) := \log^{c} \log \left( 10 + x \right)$.
\label{Prop:EstHighReg}
\end{prop}

\begin{proof}
Let $k= m + \alpha$  with $ 0 \leq \alpha < 1$ and $m$ integer. Then by the product rule (see proof in
Appendix $A$) and the Sobolev embedding (\ref{Eqn:SobolevIneq2}) we have

\begin{equation}
\begin{array}{ll}
\left\| D^{k} (u^{\lambda}  g(|u|) \right\|_{_{L_{t}^{Q}  L_{x}^{R} (J)}} & \lesssim
\| D^{k} u \|_{L_{t}^{\frac{2(n+2)}{n}} L_{x}^{\frac{2(n+2)}{n}}(J)}
\| u \|_{L_{t}^{\frac{2(n+2)}{n-2}} L_{x}^{\frac{2(n+2)}{n-2}}(J)}^{\lambda - 1} g \left(  \| u \|_{L_{t}^{\infty} \tilde{H}^{k}(J)} \right) \\
& + \| D^{k} g(|u|) \|_{L_{t}^{\frac{2(n+2)}{n}} L_{x}^{\frac{2(n+2)}{n}}(J)}
\| u \|^{\lambda - 1}_{L_{t}^{\frac{2(n+2)}{n-2}} L_{x}^{\frac{2(n+2)}{n-2}}(J)} \\
& \| u \|_{L_{t}^{\infty} \tilde{H}^{\frac{n}{2}+}(J)}
\end{array}
\label{Eqn:BoundDkugu}
\end{equation}
Let $RHS'$ be the right-hand side of (\ref{Eqn:EstHighReg}) multiplied by $\| u \|^{\lambda -1}_{L_{t}^{\frac{2(n+2)}{n-2}} L_{x}^{\frac{2(n+2)}{n-2}} (J)}$. We have

\begin{equation}
\begin{array}{ll}
\left\| D^{k} g(|u|) \right\|_{L_{t}^{\frac{2(n+2)}{n}}  L_{x}^{\frac{2(n+2)}{n}} (J)} &
\lesssim  \sum\limits_{\gamma \in \mathbb{N}^{n}: \; |\gamma| = m}  \| D^{\alpha} \partial^{\gamma} \left( g(|u|) \right)
\|_{L_{t}^{\frac{2(n+2)}{n}}  L_{x}^{\frac{2(n+2)}{n}} (J)},
\end{array}
\nonumber
\end{equation}
Let $X := \partial^{\gamma} \left( g(|u|) \right)$. Expanding we see that $X$ is a finite sum of terms
of the form $X' := \partial^{\bar{\theta}} \tilde{g}(|u|^{2}) X_{0}^{'} X_1^{'} ... X_{m}^{'}$ with

\begin{equation}
\begin{array}{ll}
X^{'}_{p} & := (\partial^{\delta_{p,1}} u)^{\theta_{p,1}}.... (\partial^{\delta_{p,q}} u)^{\theta_{p,q}}
(\partial^{\bar{\delta}_{p,1}} \bar{u})^{\bar{\theta}_{p,1}}.... (\partial^{\bar{\delta}_{p,\bar{q}}} \bar{u})^{\bar{\theta}_{p,\bar{q}}} \cdot
\end{array}
\nonumber
\end{equation}
Here $p \in \{0,1,...,m \}$, $(\delta_{p,j},\bar{\delta}_{p,j}) \in \mathbb{N}^{n} \times \mathbb{N}^{n}$,  $\bar{\theta} \in \mathbb{N}^{*}$, and $(\theta_{p,j},\bar{\theta}_{p,j}) \in \mathbb{N} \times
\mathbb{N} $ are such that $|\delta_{p,j}| = |\bar{\delta}_{p,j}| = p$ and $ \sum\limits_{l=1}^{m} l \theta'_{l} = m $
with $\theta'_{l} := \sum \limits_{j=1}^{q} \theta_{l,j} + \sum \limits_{j=1}^{\bar{q}} \bar{\theta}_{l,j} $.  \\
We prove the following claim:\\
\underline{Claim}:\\

\begin{enumerate}

\item $\theta'_{m} \in \{0, 1 \}$, and if $\theta'_{m}=1$ then  $\theta'_{m-1}=...=\theta'_{1}=0$.

\item[$$]

\item Let $l \in \{1,...,m-1 \}$. If $\theta'_{m}=...=\theta'_{m-(l-1)}=0$  and $\theta'_{m-l} \neq 0$  then
$\sum \limits_{j=1}^{m-l} \theta'_{j} \leq l+1$.

\end{enumerate}
The proof of the first statement is left to the reader. Clearly $\theta'_{m-l} \leq \frac{m}{m-l} \leq l+1$.
Hence $\theta'_{m-l} =q$ with $q \in \{1,...,l+1 \}$. From  $(m-l)q + \sum \limits_{j=1}^{m-(l+1)}j \theta'_{j} \leq m$
we get $\sum \limits_{j=1}^{m-(l+1)} \theta'_{j} \leq  l q - m(q-1)  \leq l +1 -q $., which implies that
the estimate of the second claim holds. \\
The following elementary estimates hold: \footnote{In the sequel $\tilde{H}^{k,p} := D^{-1} L^{p} \cap D^{- k} L^{p}$}
\begin{equation}
\begin{array}{l}
\left( \; \left( \left( n=4 \; \text{and} \; l \geq 2 \right) \; \text{or} \; (n=3 \; \text{and} \; l \geq 1) \right) \; \text{and}
\; 1 \leq \delta \leq l+1 \; \text{and} \; 1 \leq \bar{m} \leq m-l \; \text{and} \; 0 \leq \bar{\alpha} \leq \alpha \; \right) \; \; \text{or} \\
\left( \; (n,l)=(4,1) \; \text{and} \; 1 \leq \delta \leq 2 \; \text{and} \; 1 \leq \bar{m} \leq m-2 \; \text{and} \; 0 \leq \bar{\alpha} \leq \alpha \;  \right): \\
\left\| D^{\bar{m} + \bar{\alpha}} u  \right\|_{L_{t}^{\frac{2 \delta (n+2)}{n}} L_{x}^{\frac{2 \delta (n+2)}{n}} (J) }
\lesssim  \| u \|_{L_{t}^{\frac{2 \delta (n+2)}{n}} \tilde{H}^{k - \frac{1}{4}, \frac{2(n+2) \delta}{(n+2) \delta - 2}} (J) }; \\
\\
\left( \left( n=4 \; \text{and} \; l \in \{ 0,1 \} \right) \; \text{or} \; \left( n=3 \; \text{and} \; l=0 \right) \right) \; \text{and} \;
 1 \leq \delta \leq l+1 \; \text{and} \; 1 \leq \bar{m} \leq m-l \; \text{and} \; 0 \leq \bar{\alpha} \leq \alpha : \\
\left\| D^{\bar{m} + \bar{\alpha}} u  \right\|_{L_{t}^{\frac{2 \delta (n+2)}{n}} L_{x}^{\frac{2 \delta (n+2)}{n}} (J) }
\lesssim  \| u \|_{L_{t}^{\frac{2 \delta (n+2)}{n}} \tilde{H}^{k, \frac{2(n+2) \delta}{(n+2) \delta - 2}} (J) }; \\
\\
l \geq 2 \;  \text{and} \;  1 \leq \delta \leq l+2 \; \text{and} \; 0 \leq \bar{m} \leq m-l : \\
\\
\left\| D^{\bar{m}} u  \right\|_{L_{t}^{\frac{2 \delta (n+2)}{n}} L_{x}^{\frac{2 \delta (n+2)}{n}} (J) }
\lesssim  \| u \|_{L_{t}^{\frac{2 \delta (n+2)}{n}} \tilde{H}^{k - \frac{1}{4}, \frac{2(n+2) \delta}{(n+2) \delta - 2}} (J) };
\end{array}
\nonumber
\end{equation}
\begin{equation}
\begin{array}{l}
n=4: \left\{
\begin{array}{l}
\alpha \leq \bar{\alpha} \leq (\alpha +)+: \, \| D^{\bar{\alpha}} P_{\geq 1} u \|_{L_{t}^{\infty} L_{x}^{\infty-}(J)} \lesssim \| D^{\psi(k)} u \|_{L_{t}^{\infty} L_{x}^{2}(J)}, \\
\alpha > 0, \, 1 \leq \bar{m} \leq m-1: \; \| D^{\bar{m}} u \|_{L_{t}^{\frac{4(n+2)}{n}} L_{x}^{\frac{4(n+2)}{n}+}(J)} \lesssim
\| u \|_{L_{t}^{\frac{4(n+2)}{n}} \tilde{H}^{k,\frac{4(n+2)}{2(n+2) -2}} (J)}, \; \text{and} \\
\alpha > 0: \, \| D^{m} u \|_{L_{t}^{\frac{2(n+2)}{n}} L_{x}^{\frac{2(n+2)}{n}+} (J)} \lesssim \| u \|_{L_{t}^{\frac{2(n+2)}{n}} \tilde{H}^{k,\frac{2(n+2)}{n}}(J)} ;
\end{array}
\right. \\
n=3: \left\{
\begin{array}{l}
\| D u \|_{L_{t}^{\frac{2(n+2)}{n}} L_{x}^{\frac{15}{2}+}(J)} \lesssim
\| u \|_{L_{t}^{\frac{2(n+2)}{n}} \tilde{H}^{k,\frac{2(n+2)}{n}} (J)}, \;
\text{and}  \\
\alpha \leq \bar{\alpha} \leq (\alpha +)+:
\left\{
\begin{array}{l}
m >1 : \;  \| D^{\bar{\alpha}} P_{\geq 1} u \|_{L_{t}^{\infty} L_{x}^{\infty}(J)}  \lesssim
\| D^{\psi(k)} u \|_{L_{t}^{\infty} L_{x}^{2}(J)} \\
m=1: \; \| D^{\bar{\alpha}} P_{\geq 1} u \|_{L_{t}^{\infty} L_{x}^{6-}(J)}  \lesssim
\| D^{\psi(k)} u \|_{L_{t}^{\infty} L_{x}^{2}(J)}
\end{array}
\right.
\end{array}
\right.
\end{array}
\nonumber
\end{equation}
There exists $1  \geq \theta  \geq 0$ \footnote{In the sequel we allow the value of $\theta$ to change from one line to the other one. Here $\widehat{P_{<1}f}(\xi):= \phi (\xi) \hat{f}(\xi)$ with $\phi$ a bump function equal to one for $|\xi| \leq 1$ and supported on $|\xi| \leq 2$ and $\widehat{P_{\geq 1} f}(\xi) := \hat{f}(\xi) - \widehat{P_{<1}f}(\xi) $.} such that

\begin{equation}
\begin{array}{ll}
k' > 0, \; \delta \in \{1,...,l+2 \}: \; \| u  \|_{L_{t}^{\frac{2 \delta(n+2)}{n}} \tilde{H}^{k',\frac{2 \delta(n+2)}{\delta(n+2)-2}} (J)}  \lesssim \| u \|^{\theta}_{L_{t}^{\frac{2(n+2)}{n}} \tilde{H}^{k', \frac{2(n+2)}{n}} (J)}
\| u \|^{1- \theta}_{L_{t}^{\infty} \tilde{H}^{k'} (J)} \\
0 \leq \bar{\alpha} \leq 1: \; \| D^{\bar{\alpha}} P_{< 1} u \|_{L_{t}^{\infty} L_{x}^{\infty}(J)} \lesssim  \| D u \|_{L_{t}^{\infty} L_{x}^{2}(J)}  \cdot
\end{array}
\nonumber
\end{equation}
$\left\| D^{k} ( g(|u|) \right\|_{L_{t}^{\frac{2(n+2)}{n}}  L_{x}^{\frac{2(n+2)}{n}} (J)}$
is bounded by a finite sum of terms of the form  \footnote{In the sequel if some terms do not make sense
we do not take them into account. Example: if $p=0$ then one should not take into account
the term where $D^{p-1}$ appears. Also if $\theta'_{j} =0$ for some $j \in \{0,..,m-l\}$ then
we ignore all the terms where $j$ appears.}

\begin{equation}
\begin{array}{l}
\bar{Y}_{p}:=  \| u \|^{\theta'_0}_{L_{t}^{q_0} L_{x}^{r_0}(J)}... \| D^{p-1} u \|_{L_{t}^{q_{p-1}}
L_{x}^{r_{p-1}}(J)}^{\theta'_{p-1}}  \\
\| D^{p} u \|_{L_{t}^{q_{p}} L_{x}^{r_{p}}(J)}^{\theta'_{p}-1}
\| D^{ \alpha + p} u \|_{L_{t}^{\bar{q}_{p}} L_{x}^{\bar{r}_{p}}(J)}  \\
\| D^{p+1} u \|^{\theta'_{p+1}}_{L_{t}^{q_{p+1}} L_{x}^{r_{p+1}}(J)} ...
\| D^{m-l} u \|^{\theta'_{m-l}}_{L_{t}^{q_{m-l}} L_{x}^{r_{m-l}}(J)}
\| \partial^{\bar{\theta}} \tilde{g}(|u|^{2}) \|_{L_{t}^{q'} L_{x}^{r'}(J)},
\end{array}
\nonumber
\end{equation}
(with $p \in \{0,...,m-l\}$, $\sum \limits_{j=0: j \neq p}^{m-l} \frac{\theta'_j}{q_j} + \frac{\theta'_p -1}{q_p} +
\frac{1}{\bar{q}_p} + \frac{1}{q'} = \frac{n}{2(n+2)}$,
 and $\sum \limits_{j=0: j \neq p}^{m-l} \frac{\theta'_j}{r_j}
 + \frac{\theta'_p -1}{r_p} + \frac{1}{\bar{r}_p} + \frac{1}{r'} = \frac{n}{2(n+2)}$ ), and of the form
\begin{equation}
\begin{array}{ll}
\tilde{Y} & := \| u \|^{\theta'_0}_{L_{t}^{q_{0}} L_{x}^{r_0}(J)} ...
\| D^{m-l} u \|^{\theta'_{m-l}}_{ L_{t}^{q_{m-l}} L_{x}^{r_{m-l}}(J)}
\left\| D^{\alpha} \partial^{\bar{\theta}} \tilde{g}(|u|^{2}) \right\|_{L_{t}^{q'} L_{x}^{r'}(J)}
\end{array}
\nonumber
\end{equation}
(with $\sum \limits_{j=0}^{m-l} \frac{\theta'_j}{q_j} + \frac{1}{q'} =\frac{n}{2(n+2)}$ and $\sum \limits_{j=0}^{m-l} \frac{\theta'_j}{r_j} + \frac{1}{r'}
=\frac{n}{2(n+2)}$ ). Here $l \in \{0,...,m-1 \}$ and $\theta'_{m-l} \neq 0$. \\
\\
We first estimate $\bar{Y}_{p}$. \\
Assume that $p \neq 0$. \\
Then with $q_0=r_0 = \infty $ and $ q_j =r_j= \bar{q}_p = \bar{r}_p  = \frac{2(n+2)}{n} \sum \limits_{s=1}^{m-l} \theta'_{s}$ for $j \neq 0$ we see from elementary estimates of the derivatives of $g$, the above claim, the above estimates, and elementary consequences of
$ \sum\limits_{l=1}^{m} l \theta'_{l} = m $  that $\bar{Y}_{p}$ is bounded by $RHS'$. \\
Assume that $p=0$. \\
We first consider the case where $(n,l) \neq (3,0)$,  $(n,l) \neq (3,1)$, $(n,l) \neq (4,0)$, and $(n,l) \neq (4,1)$. Letting $q_0 = r_0 = \infty$ and $\bar{q}_0 = \bar{r}_0 =q_j=r_j =\frac{2(n+2)}{n}  \left( \sum \limits_{s=1}^{m-l} \theta'_{s} + 1 \right)$ for
$j \neq 0$ we see that $\bar{Y}_{p} $ is bounded by $RHS'$.\\
We then consider the other cases. By decomposition one has to estimate \\
$Z_{lo}:= \| u \|^{\theta'_0 -1}_{L_{t}^{q_0} L_{x}^{r_0}(J)} \| D^{\alpha} P_{< 1} u \|_{L_{t}^{\bar{q}_0} L_{x}^{\bar{r}_0}(J)} ...
\| D^{m-l} u \|^{\theta'_{m-l}}_{L_{t}^{q_{m-l}} L_{x}^{r_{m-l}}(J)} \| \partial^{\bar{\theta}} \tilde{g}(|u|^{2}) \|_{L_{t}^{q'} L_{x}^{r'}(J)} $ and
$Z_{hi}:=\| u \|^{\theta'_0 -1}_{L_{t}^{q_0} L_{x}^{r_0}(J)} \| D^{\alpha} P_{ \geq 1} u \|_{L_{t}^{\bar{q}_0} L_{x}^{\bar{r}_0}(J)} ...
\| D^{m-l} u \|^{\theta'_{m-l}}_{L_{t}^{q_{m-l}} L_{x}^{r_{m-l}}(J)} \| \partial^{\bar{\theta}} \tilde{g}(|u|^{2}) \|_{L_{t}^{q'} L_{x}^{r'}(J)} $.\\
Assume that $(n,l)=(3,0)$. Letting $(q_{0},r_{0}) = (\bar{q}_{0},\bar{r}_{0}) = (\infty, \infty)$ and
$q_{m} = r_{m} = \frac{2(n+2)}{2}$ we see that $Z_{lo}$ is bounded by $RHS^{'}$. Letting
$q_0 = \bar{q}_0 = r_0 = \infty$, $\bar{r}_0 = 6-$ if $m = 1$ (resp. $\bar{r}_0 = \infty$ if $m >1$),
$q_m = \frac{2(n+2)}{n}$, and $r_m = \frac{15}{2} +$ if $m=1$ (resp. $r_m = \frac{2(n+2)}{n}$ if $m >1$), we see that
$Z_{hi}$ is bounded by $RHS'$. Assume that $(n,l)=(4,0)$. $Z_{lo}$ is bounded by $RHS'$, by assigning the same values
to the exponents as for the case $(n,l)=(3,0)$. Let $(q_0,r_0)= (\infty,\infty)$. If $\alpha >0$ (resp. $\alpha =0$) let
$(\bar{q}_0,\bar{r}_0) = (\infty,\infty-) $ (resp. $(\bar{q}_0,\bar{r}_0) = (\infty,\infty)$) and
$(q_m,r_m) = \left( \frac{2(n+2)}{n}, \frac{2(n+2)}{n} + \right)$ $ \left( \textrm{resp.} \left( \frac{2(n+2)}{n}, \frac{2(n+2)}{n} \right) \right)$: this
implies that $Z_{hi}$ is bounded by $RHS'$. Assume now that $(n,l)=(3,1)$ or $(4,1)$. Consider the subcase
$(\theta'_{1},...,\theta'_{m-2},\theta'_{m-1})=(0,...,0,2)$ (resp. $(\theta'_{1},...,\theta'_{m-2},\theta'_{m-1})=(1,...,0,1)$ ).
Letting $r_{m-1} = \frac{4(n+2)}{n-2}$ (resp. $r_{1} = r_{m-1} = \frac{4(n+2)}{n-2}$),
$(\bar{q}_0,\bar{r}_0) = (\infty,\infty)$, we see that $Z_{lo}$ is bounded by $RHS'$. Consider the subcase
$(\theta'_{1},...,\theta'_{m-2},\theta'_{m-1})=(0,...,0,2)$. If $(n,l) = (4,1)$ and $\alpha > 0$ (resp. $\alpha=0$)
let $(q_{m-1},r_{m-1}) = \left( \frac{4(n+2)}{n}, \frac{4(n+2)}{n} + \right)$ and
$(\bar{q}_0, \bar{r}_0) = (\infty,\infty-)$ $\left( \textrm{resp.} (q_{m-1},r_{m-1}) = \left( \frac{4(n+2)}{n}, \frac{4(n+2)}{n} \right) \,
\textrm{and} \, (\bar{q}_0, \bar{r}_0) = (\infty,\infty) \right)$. If $(n,l) = (3,1)$  let
$(q_{m-1}, r_{m-1}) = \frac{4(n+2)}{n} (1,1)$ and $(\bar{q}_0,\bar{r}_0) = (\infty,\infty)$. This implies that $Z_{hi}$ is bounded by
$RHS'$. Now consider the subcase $(\theta'_{1},...,\theta'_{m-2},\theta'_{m-1})=(1,...,0,1)$: this subcase is treated similarly, except that
we assign the same value of $r_{m-1}$ (resp. $q_{m-1}$) of the previous subcase to that of the variable $r_1$ (resp. $q_{1}$).\\
\\
We then estimate $\tilde{Y}$. \\
Writing $ f = P_{< 1} f + P_{\geq 1} f $, we see that given $p \geq 1$,
$\| D^{\alpha} f \|_{L^{p}(\mathbb{R}^{n})} \lesssim  \| f \|_{L^{p}(\mathbb{R}^{n})} +
\| f \|_{\dot{B}^{\alpha+}_{p,p}(\mathbb{R}^{n})}$ and
$\| f \|_{\dot{B}^{\alpha+}_{p,p}(\mathbb{R}^{n})} \lesssim \| f \|_{L^{p} (\mathbb{R}^{n})} +
\| D^{(\alpha+)+} f \|_{L^{p}(\mathbb{R}^{n})} $. Here $\dot{B}^{\alpha+}_{p,p}(\mathbb{R}^{n})$ is the standard homogeneous Besov space. Elementary estimates show that
$ \| \partial^{\bar{\theta}} \tilde{g}(|u|^{2})(x+h) - \partial^{\bar{\theta}} \tilde{g}(|u|^{2})(x) \|_{L^{p} (\mathbb{R}^{n})} \lesssim
\| u(x+h) - u(x) \|_{L^{p}(\mathbb{R}^{n})} \langle \bar{Q}_{\frac{n}{2}+} (J,u) \rangle^{C} $ for some constant $C >0$. Hence from the characterization of the Besov norm by the modulus of continuity, we see that
$\| D^{\alpha} \partial^{\bar{\theta}} \tilde{g}(|u|^{2}) \|_{L^{p}(\mathbb{R}^{n})} \lesssim
\| \partial^{\theta} \tilde{g}(|u|^{2}) \|_{L^{p}(\mathbb{R}^{n})}
+ \| D^{(\alpha+)+} u \|_{L^{p}(\mathbb{R}^{n})}$. \\
Hence one has to estimate \\
$\tilde{Y}_{1} := \| u \|^{\theta'_0}_{L_{t}^{q_0} L_{x}^{r_0} (J)} ... \| D^{m-l} u \|^{\theta'_{m-l}}_{L_{t}^{q_{m-l}} L_{x}^{r_{m-l}}(J)}
\| \partial^{\theta} \tilde{g}(|u|^{2} \|_{L_{t}^{q'} L_{x}^{r'}(J)}$, and \\
$ \tilde{Y}_{2} := \| u \|^{\theta'_0}_{L_{t}^{q_0} L_{x}^{r_0} (J)} ... \| D^{m-l} u \|^{\theta'_{m-l}}_{L_{t}^{q_{m-l}} L_{x}^{r_{m-l}}(J)}
\| D^{(\alpha+)+} u \|_{L_{t}^{q'} L_{x}^{r'}(J)}$.
We write $\tilde{Y}_2= \tilde{Y}_{2,lo} + \tilde{Y}_{2,hi}$ with \\
$\tilde{Y}_{2,lo}: = \| u \|^{\theta'_0}_{L_{t}^{q_0} L_{x}^{r_0} (J)} ... \| D^{m-l} u \|^{\theta'_{m-l}}_{L_{t}^{q_{m-l}} L_{x}^{r_{m-l}}(J)}
\| D^{(\alpha+)+} P_{<1} u \|_{L_{t}^{q'} L_{x}^{r'}(J)}$ and
$\tilde{Y}_{2,hi}: = \| u \|^{\theta'_0}_{L_{t}^{q_0} L_{x}^{r_0} (J)} ... \| D^{m-l} u \|^{\theta'_{m-l}}_{L_{t}^{q_{m-l}} L_{x}^{r_{m-l}}(J)}
\| D^{(\alpha+)+} P_{\geq 1} u \|_{L_{t}^{q'} L_{x}^{r'}(J)}$. \\
We first consider the case when $(n,l) \neq (3,0)$, $(n,l) \neq (3,1)$, $(n,l) \neq (4,0)$, and $(n,l) \neq (4,1) $. We can estimate
$\tilde{Y}_{1}$  by $RHS'$, assigning the same values for $q_0$,$r_0$,...,$q_{m-l}$, $r_{m-l}$ (resp. $q'$,$r'$) as
those for the same exponents (resp. $\bar{q}_{0}$, $\bar{r}_{0}$) when we estimated $\bar{Y}_{0}$.
We can estimate $\tilde{Y}_{2,lo}$ (resp. $\tilde{Y}_{2,hi}$) by $RHS'$, assigning the same values for
$q_0$, $r_0$,...$q_{m-l}$, $r_{m-l}$ (resp. $q^{'},r^{'}$) as those for the same exponents (resp. $\bar{q}_{0},\bar{r}_{0}$)
when we estimated  $Z_{lo}$ (resp. $Z_{hi}$, with $\alpha > 0$). We then consider the other cases. We can estimate
$\tilde{Y}_{2,lo}$ and $\tilde{Y}_{1}$ (resp. $\tilde{Y}_{2,hi}$ ) by $RHS'$, assigning
the same values for $q_0$, $r_0$,...,, $q_{m-l}$, $r_{m-l}$ (resp. $q'$, $r'$) as those for the same exponents (resp. $\bar{q}_0$, $\bar{r}_0$)
when we estimate $Z_{lo}$ (resp. $Z_{hi}$). \\
\\
A straightforward modification of the proof shows that (\ref{Eqn:EstHighReg}) holds if $u^{\lambda}$ is replaced with
$u^{\lambda_1} \bar{u}^{\lambda_2}$ with $(\lambda_1,\lambda_2) \in \mathbb{N}^{2}$ such that $\lambda_1 + \lambda_2 = \lambda$. \\
If we replace $g(|u|)$ with $\tilde{g}^{'}(|u|^{2}) u^{\lambda_3} \bar{u}^{\lambda_4}$, then (\ref{Eqn:EstHighReg}) also holds
by replacing in the proof $\partial^{\bar{\theta}} \tilde{g}(|u|^{2})$ with $\partial^{\bar{\theta}+1} \tilde{g}(|u|^{2}) u^{\lambda_a}
\bar{u}^{\lambda_b}$, taking into account that $\left\| \partial^{\bar{\theta}+1} \tilde{g}(|u|^{2}) u^{\lambda_{a}} \bar{u}^{\lambda_{b}} \right\|_{L^{r}(\mathbb{R}^{n})} \lesssim \left\| \partial^{\bar{\theta}} \tilde{g}(|u|^{2}) \right\|_{L^{r} (\mathbb{R}^{n})}$ and
$ \left\| D^{\alpha} \left( \partial^{\bar{\theta}+1} \tilde{g}(|u|^{2}) u^{\lambda_a} \bar{u}^{\lambda_b}  \right) \right\|_{L^{r}(\mathbb{R}^{n})}
\lesssim  \| \partial^{\bar{\theta}+1} \tilde{g}(|u|^{2}) u^{\lambda_a} \bar{u}^{\lambda_b} \|_{L^{r}(\mathbb{R}^{n})} + \left\| D^{(\alpha+)+} u \right\|_{L^{r}(\mathbb{R}^{n})}$ ( Here $r \in [1,\infty]$ and $(\lambda_a,\lambda_b) \in \mathbb{N}^{2}$ such that $\lambda_a + \lambda_b=2$). The proof is left to the
reader.

\end{proof}

\section{APPENDIX C}

We shall prove the following proposition:

\begin{prop}
Let $u$ be a solution of (\ref{Eqn:BarelySchrod}) with data
$u_0 \in \tilde{H}^{k}$, $k > \frac{n}{2}$. Assume that $u$ exists globally in time and that
$\| u \|_{L_{t}^{\frac{2(n+2)}{n-2}} L_{x}^{\frac{2(n+2)}{n-2}} (\mathbb{R})} < \infty$. Then
$Q(\mathbb{R},u) < \infty$.
\label{Prop:PersReg}
\end{prop}

\begin{proof}

By symmetry we may WLOG restrict ourselves to $\mathbb{R^{+}}$. \\
First assume that $ \frac{n+2}{n-2} > k > \frac{n}{2}$. Repeating the same steps from `$J:=[0,a]$ just below
(\ref{Eqn:ControlApCrit}) to (\ref{Eqn:EstQPr}) included and replacing `$[0,T']$' (resp. `LHS \; of \; (\ref{Eqn:EstI}) $\gtrsim$')
with `$\mathbb{R}^{+}$' (resp. `$\infty >$') we get $Q(\mathbb{R}^{+},u) < \infty$.\\
Now assume that $k \geq \frac{n+2}{n-2}$. In view of the previous paragraph it is sufficient to show that if
for all $1 \leq j \leq k - \frac{1}{4}$, $Q_{j}(\mathbb{R}^{+},u) < \infty$, then $Q(\mathbb{R}^{+},u) < \infty$. Let
$J:=[0,a]$ be an interval. By (\ref{Eqn:Strich}), (\ref{Eqn:SobolevIneq2}), Proposition \ref{Prop:FracLeibn} and Proposition \ref{Prop:EstHighReg}, we get

\begin{equation}
\begin{array}{ll}
Q(J,u) & \lesssim \| u_{0} \|_{\tilde{H}^{k}} + Q(J,u) \| u \|^{\frac{4}{n-2}}_{L_{t}^{\frac{2(n+2)}{n-2}} L_{x}^{\frac{2(n+2)}{n-2}} (J)}
g \left( Q(J,u) \right) \cdot
\end{array}
\nonumber
\end{equation}
Again repeating the same steps from `$J:= [0,a]$' just below (\ref{Eqn:ControlApCrit}) to
(\ref{Eqn:EstQPr}) included, and taking into account the replacements that were pointed out for the case
$ \frac{n+2}{n-2} > k > \frac{n}{2}$, we get $Q(\mathbb{R}^{+},u) < \infty$.

\end{proof}








\begin{thebibliography}{1}

\bibitem{bour} J. Bourgain, \emph{Global well-posedness of defocusing 3D critical NLS in the radial case}, JAMS 12 (1999), 145-171

\bibitem{cazweiss} T. Cazenave, F.B. Weissler, \emph{The Cauchy problem for the critical nonlinear Schr\"odinger equation in $H^{s}$},
Nonlinear. Anal. 14 (1990), no. 10, 807-836.

\bibitem{christwein} M. Christ and M. Weinstein, \emph{Dispersion of small amplitude solutions of the generalized Korteweg-de Vries
equation}, J .Funct. Analysis 100 (1991), 87-109

\bibitem{collkeelstafftaktao}  J. Colliander, M. Keel, G. Staffilani, H. Takaoka, T. Tao, \emph{Global well-posedness and scattering in the energy
space for the critical nonlinear Schr\"odinger equation in $\mathbb{R}^{3}$}, Annals of Math. 167 (2008), 767-865

\bibitem{grill} M. Grillakis, \emph{On nonlinear Schr\"odinger equations}, Comm. Partial Differential Equations 25 (2000), no 9-10, 1827-1844

\bibitem{keeltao} M. Keel and T. Tao, \emph{Endpoint Strichartz estimates}, Amer. Math. J. 120 (1998), 955-980
\bibitem{rickmanvisan} E. Rickman and M. Visan, \emph{Global well-posedness and scattering for the defocusing energy-critical nonlinear
Schrodinger equation in $R^{1+4}$}, Amer. J. Math. 129 (2007), 1-60

\bibitem{taylor} M. Taylor, \emph{Tools for PDE. Pseudodifferential operators, paradifferential operators, and layer Potentials}. Mathematical
Surveys and Monographs, 81. American Mathematical Society,Providence, RI, 2000.

\bibitem{triroybar} T. Roy, \textit{One remark on barely $\dot{H}^{s_{p}}$ supercritical
wave equations }, preprint, arXiv:0906.0044

\bibitem{taorad} T. Tao, \emph{Global well-posedness and scattering for the higher-dimensional energy-critical non-linear Schr\"odinger
equation for radial data}, New York J. Math., 11, 2005, 57-80

\bibitem{visan} M.Visan, \emph{The defocusing energy-critical nonlinear Schrodinger equation in higher dimensions}, Duke Math. J. 138 (2007), 281-374.

\end{thebibliography}
\end{document}